\documentclass[12pt]{article}

\usepackage[colorlinks=true,citecolor=black,linkcolor=black,urlcolor=blue]{hyperref}
\newcommand{\arxiv}[1]{\href{http://arxiv.org/abs/#1}{\texttt{arXiv:#1}}}

\usepackage{amsmath,amssymb,amsthm}
\usepackage{tikz,color,graphicx}

\numberwithin{equation}{section}

\newcommand{\remove}[1]{}

\newtheorem{theorem}{Theorem}[section]
\newtheorem{lemma}[theorem]{Lemma}

\newtheorem{proposition}[theorem]{Proposition}
\newtheorem{definition}[theorem]{Definition}

\def\nfrac#1#2{{\textstyle\frac{#1}{#2}}}
\def\dfrac#1#2{\lower0.15ex\hbox{\large$\frac{#1}{#2}$}}
\def\calH{{\mathcal H}}
\def\calL{{\mathcal U}}

\newcommand{\avec}{\boldsymbol{a}}
\newcommand{\dvec}{\boldsymbol{d}}
\newcommand{\xvec}{\boldsymbol{x}}
\newcommand{\vecdvec}{(\dvec^-,\dvec^+)}
\newcommand{\Omegau}{\Omega(\dvec)}
\newcommand{\Omegad}{\Omega(\dvec^-,\dvec^+)}

\newcommand{\Md}{\mathcal{M}(\dvec^-,\dvec^+)}

\newcommand{\Condone}{{\tt Condition\!~1}}
\newcommand{\Condtwo}{{\tt Condition\!~2}}

\newcommand{\norm}[1]{{\|#1\|}}

\makeatletter

\renewcommand{\p@enumii}{}

\renewcommand{\p@enumiii}{}
\makeatother

\def\eps{\epsilon}
\def\N{{\mathcal N}}

\title{Mixing time of the switch Markov chain\\ and stable degree sequences}
\author{Pu Gao\thanks{Research supported by NSERC.}\\
\small Department of Combinatorics and Optimization\\[-0.2ex]
\small University of Waterloo\\[-0.2ex]
\small ON, N2L 3G1, Canada\\
\small pu.gao@uwaterloo.ca
\and Catherine Greenhill\thanks{Research supported by the Australian Research Council
Discovery Project DP190100977.}\\
\small School of Mathematics and Statistics\\[-0.2ex]
\small  UNSW Sydney\\[-0.2ex]
\small NSW 2052, Australia\\
 \small c.greenhill@unsw.edu.au
 }
\date{30 November 2020}

\begin{document}

\maketitle

\begin{abstract}
The switch chain is a well-studied Markov chain which can be used to sample 
approximately uniformly from the set $\Omega(\dvec)$ of all graphs
with a given degree sequence~$\dvec$.  Polynomial mixing time 
(rapid mixing)
has been established for the switch chain under various conditions on the degree
sequences. Amanatidis and Kleer introduced the notion of strongly stable
families of degree sequences, and proved that the switch chain is rapidly
mixing for any degree sequence from a strongly stable family.  Using a different
approach, Erd{\H o}s et al.~recently extended this result to the (possibly larger)
 class
of P-stable degree sequences, introduced by Jerrum and Sinclair in 1990.
We define a new notion of stability for a given degree sequence, namely
$k$-\emph{stability}, and prove that if a degree sequence $\dvec$ is 8-stable then the
switch chain on $\Omega(\dvec)$ is rapidly mixing.  
We also provide sufficient conditions
for P-stability, strong stability and 8-stability. Using these sufficient conditions,
we give the first proof of P-stability for various families of 
heavy-tailed degree sequences, including power-law degree sequences,
and show that the switch chain is rapidly mixing for these families. 

We further extend these notions and results to directed degree sequences.
\end{abstract}

\section{Introduction}

Given a finite set of discrete objects $\Omega$ and a distribution $\pi$ over $\Omega$, how can we efficiently sample an object from $\Omega$ according to distribution $\pi$? This is a classical problem in theoretical computer science, with many applications in other research fields such as statistics, engineering, and different branches of sciences. In most situations, $\Omega$ is a very large set, and it is not possible, given limited time and computation power, to enumerate all objects in $\Omega$ and compute the probability of each object under $\pi$.  

There are various general methods developed to solve this problem. The most commonly applied method is Markov Chain Monte Carlo (MCMC). Using MCMC, one needs to define a Markov chain with state space $\Omega$ and stationary distribution $\pi$. Then, output an object in $\Omega$ after running the Markov chain sufficiently long time. The MCMC-based algorithms are approximate samplers. The challenge is to obtain an upper bound on the so-called mixing time,  the minimum number of steps required to run the Markov chain so that the distribution of the output differs from $\pi$ by a prescribed tolerance, say $\eps$, in total variation distance. The techniques used to bound the mixing time are problem-specific.

In this paper, we consider the problem of uniformly sampling a graph on vertex set $[n]=\{1,2,\ldots, n\}$ with a specified degree sequence, where $n\ge 1$ is a positive integer.  A \emph{degree sequence} is a sequence $\dvec=(d_1,\ldots, d_n)$ of nonnegative
integers with even sum.  
We say $\dvec$ is {\em graphical} if there exists a simple graph on the vertex set $[n]$ 
such that vertex $i$ has degree $d_i$ for every $i\in [n]$. Let $\Omega(\dvec)$ be the set of all graphs with vertex set $[n]$
and degree sequence $\dvec$. Hence, we will study the problem with $\Omega=\Omega(\dvec)$ and with $\pi$ the uniform distribution over $\Omega(\dvec)$.
This problem has many practical applications for researchers who use graphs to model complex
discrete systems.

The first MCMC approach to uniformly sampling graphs with a given degree was
given by Jerrum and Sinclair~\cite{JS90} in 1990. They defined a Markov chain to perform
approximately uniform sampling from a set of graphs $\Omega'(\dvec)$
which contains $\Omega(\dvec)$.  Every graph in $\Omega'(\dvec)$ has degree
sequence which is either $\dvec$, or a very small perturbation of $\dvec$.
Rejection sampling is performed until the output of the Markov chain belongs to $\Omega(\dvec)$.
Hence the expected runtime of their algorithm is polynomial
precisely when $\Omega(\dvec)$ is at least a polynomial fraction of 
$\Omega'(\dvec)$.  
Jerrum and Sinclair introduced the notion of \emph{P-stability}
to capture this condition.
 See Section~\ref{sec:comparison} for a precise definition of P-stability.

The switch Markov chain $\mathcal{M}(\dvec)$ (or \emph{switch chain})
has state space $\Omega(\dvec)$ and makes a transition by replacing a pair of edges 
by another pair of edges (possibly the same pair), ensuring that the resulting graph 
is simple and has degree sequence $\dvec$.  A more formal definition is given in 
Section~\ref{s:switch-chain}.  This chain is appealing because, unlike the
Jerrum--Sinclair chain, every sample belongs to $\Omega(\dvec)$.
Many authors have studied the mixing time of the switch chain for
particular families of degree sequences, bipartite degree sequences and
directed degree sequences~\cite{CDG,directed,EKMS, EMMS,EMT, GS,KTV,MES}. 
Amanatidis and Kleer~\cite{AK} gave an ingenious comparison argument to prove that
the switch chain has polynomial mixing time for \emph{strongly stable}
classes of degree sequences. 
(See Section~\ref{sec:comparison} for a precise definition of strongly stability.)
The classes of degree sequences for which the switch chain 
for graphs was known to be rapidly mixing (before~\cite{AK})
are all strongly stable,
so the theorem of~\cite{AK} can be seen as a common generalisation of these
results. Amanatidis and Kleer also adapt the definition of strongly stable
to classes of bipartite degree sequences, and prove that the switch chain for 
bipartite graphs has polynomial mixing time for strongly stable classes of bipartite
degree sequences.  This gives a common generalisation of the results 
in~\cite{EMMS,KTV,MES} for the bipartite switch chain.

All strongly stable degree classes are P-stable, but it is not known whether
the converse holds.  Recently, Erd{\H o}s et al.~\cite{hungarians} proved that the
switch chain has polynomial mixing time (rapid mixing) for all P-stable  degree
sequences (also in the bipartite and directed setting).  Hence this result
extends that of~\cite{AK} from strongly stable to P-stable classes, and to
directed degree sequences.
However, P-stability is not a necessary condition for rapid mixing of the switch
chain, as shown recently by Erd{\H o}s et al.~\cite{non-P-stable}.

\bigskip

In this paper, we define a new notion of stability for degree sequences. 
For any vector $\xvec\in\mathbb{R}^n$ let $\norm{\xvec}_1 = \sum_{j=1}^n |x_j|$ denote the
$\ell_1$-norm of $\xvec$. Necessarily, for any graphical degree sequence $\dvec$, we have $\dvec\in {\mathbb N}^n$ and $\norm{\dvec}_1$ even. Let $M(\dvec)$ denote 
$\norm{\dvec}_1$.

\begin{definition}\label{def:k-stability}
Given a positive integer $k$ and nonnegative real number $\alpha$, 
we say a graphical degree sequence $\dvec$ is $(k,\alpha)$-\emph{stable} if
\[ |\Omega(\dvec')| \leq  M(\dvec)^{\alpha}\, |\Omega(\dvec)|\]
for every graphical sequence $\dvec'$ with $\norm{\dvec' - \dvec}_1 \leq k$.
Let
$\mathcal{D}_{k,\alpha}$ be the set of all degree sequences that are $(k,\alpha)$-stable. 
We say that a family $\mathcal{D}$ of degree sequences is $k$-\emph{stable} if
there exists a constant $\alpha>0$ such that 
${\mathcal D}\subseteq  {\mathcal D}_{k,\alpha}$.
\end{definition}

Obviously,
\begin{align}
\text{ for all $k$},\quad & {\mathcal D}_{k,\alpha_1}\subseteq {\mathcal D}_{k,\alpha_2} \quad  \mbox{if $\alpha_1\le \alpha_2$,} \label{alpha}\\ 
\text{ for all $\alpha$},\quad & {\mathcal D}_{k_1,\alpha} \subseteq {\mathcal D}_{k_2,\alpha} \quad \mbox{if $k_1\ge k_2$.} \label{k}
\end{align}

Our definition of 2-stability is equivalent to the notion of P-stability, 
introduced by Jerrum and Sinclair~\cite{JS90}. We prove this in Proposition~\ref{p:equivalence}. 
It is not clear whether $k$-stability is equivalent to P-stability for fixed $k>2$. 
Relationships between P-stability, strong stability and 8-stability will be further discussed in Section~\ref{sec:comparison}.

Suppose that ${\mathcal D}$ is an 8-stable family of degree sequences.  One of our main results is the following:
\begin{equation}
\mbox{The switch chain on $\Omega(\dvec)$ mixes in polynomial time for all $\dvec\in {\cal D}$.}\label{main}
\end{equation}
A more accurate statement is given in Theorem~\ref{thm:mixing}.

We are usually interested in \emph{sequences} of degree sequences
$\dvec(n)$, indexed by the positive integers $n\in {\mathbb Z}_{\ge 1}$, or $n\in {\cal I}$ where ${\cal I}$ is an infinite subset of ${\mathbb Z}_{\ge 1}$. 
Due to the technical nature of the definition of $P$-stability, strong stability and 
$k$-stability, it is in general not easy to determine if a sequence of degree sequences $(\dvec(n))_{n\in {\cal I}}$ is stable. Jerrum, Sinclair and McKay~\cite{JMS} gave 
two sufficient conditions for a family of degree sequences to be $P$-stable.
Their first condition,~\cite[Theorem~8.1]{JMS},
 is an inequality involving the maximum and minimum degrees
and $n$, while their second condition~\cite[Theorem~8.3]{JMS} is a ``more refined''
inequality which also involves the number of edges.
Using~\cite[Theorem~8.1]{JMS},  they verified that the family of regular degree sequences is $P$-stable. They also verified $P$-stability of several other families of ``moderate'' degree sequences, in the sense that the degrees are not too far from being regular. 

In this paper, we will give a sufficient condition, \Condone, for a degree
sequence to be 8-stable. We will also give a sufficient condition, \Condtwo, 
for $P$-stability and strong stability. \Condtwo\ is slightly weaker than \Condone. While our conditions apply to moderate degree sequences, they work particularly well for heavy-tailed degree sequences.   
We will prove that the following families of degree sequences all satisfy \Condone\ (and thus also \Condtwo), and thus they are ``stable'' in the sense of P-stability, strong stability and 8-stability:
\begin{itemize}
\item power-law degree sequences with exponent greater than 2;
\item bi-regular degree sequences permitting both constant degrees and degrees linear in~$n$;
\item other heavy-tailed degree sequences examined in~\cite{GW-power}.
\end{itemize}

\bigskip

Our theorems apply to a possibly smaller class of degree sequences than those 
of~\cite{hungarians}. 
(See Section~\ref{sec:comparison}
for a comparison between 8-stability, P-stability and strong stability.)
However, our approach is much simpler and requires less additional
machinery than the arguments of~\cite{AK,hungarians}. 
We make minimal changes to the multicommodity flow argument
given in~\cite{CDG,GS}, presenting a new counting argument to prove the 
so-called ``critical lemma''.
Furthermore, our sufficient conditions \Condone, \Condtwo\ enable us to 
expand the class of degree sequences known to be P-stable, subsuming the sufficient 
condition from~\cite[Theorem~1.1]{GS}. 
\medskip

\noindent {\bf Remark.}\ 
Power-law degree sequences, and other heavy-tailed degree sequences,
are of particular importance in network science.
The more refined sufficient condition for P-stability given 
in~\cite[Theorem~8.3]{JMS} can be used to show that the family of 
``power-law distribution-bounded'' degree sequences (defined in 
Section~\ref{ss:power-distribution}) is P-stable when the average degree is at 
most twice the minimum degree.   The requirement that the average degree is at 
most twice the minimum degree effectively excludes most of the interesting power-law degree sequences, i.e.\ power-law degree sequences with minimum degree 1 and average degree greater than 2.
The first author and Wormald mentioned in~\cite{Gao-power} that power-law degree sequences with exponent greater than 2 can be shown to be P-stable. However, that assertion has not been proved, in~\cite{Gao-power} or elsewhere. We establish the 8-stability, and 
thus the P-stability, of power-law degree sequences in Section~\ref{sec:applications}. 
\bigskip

To conclude this section, we discuss approaches to sampling from $\Omega(\dvec)$
which are not based on Markov chains.
McKay and Wormald~\cite{McKW90} gave an algorithm, based on an operation
called switchings, which performs \emph{exactly}
uniform sampling from  $\Omega(\dvec)$ in expected polynomial time
when the maximum degree is not too large.  This result has been improved
upon by Wormald and the first author~\cite{GW}, who achieved expected runtime $O(d^3 n)$
for $d$-regular degree sequences when $d=o(\sqrt{n})$. 
Wormald and the first author also adapted
their approach to power law degree sequences\cite{Gao-power}. 
Very recently these results were improved further by Arman, Wormald and the first author~\cite{AGW},
who presented an algorithm with expected runtime which is $O(nd + d^4)$ for
$d$-regular graphs with $d=o(\sqrt{n})$, and is $O(n)$ for the same
class of power-law degree sequences considered in~\cite{Gao-power}
(roughly speaking, those with power-law exponent greater than $2.88$).

Switchings-based approximate sampling algorithms with very fast expected
runtime (linear or sub-quadratic) have been given by Bayati, Kim and Saberi~\cite{BKS},
Kim and Vu~\cite{KV}, Steger and Wormald~\cite{SW} and Zhao~\cite{zhao}.
The error in the output distribution of these algorithms are functions of $n$ and
tend to zero as $n$ grows.  Unlike Markov chain algorithms, this error cannot be
made smaller by increasing the runtime of the algorithm.

\section{Main results}\label{sec:results}

Given $\dvec$, let $\mu\in S_n$ be a permutation such that 
$d_{\mu(1)}\ge d_{\mu(2)}\ge \cdots d_{\mu(n)}$.
Define
\[
\Delta(\dvec)=d_{\mu(1)},\quad J(\dvec)=\sum_{i=1}^{d_{\mu(1)}} d_{\mu(i)}.
\]
That is, $\Delta(\dvec)$ is the largest entry of $\dvec$, while
$J(\dvec)$ is the sum of the $\Delta(\dvec)$ largest entries of~$\dvec$.

First we state~(\ref{main}) in a more formal and accurate form.

\begin{theorem}\label{thm:mixing}
Suppose that the graphical degree sequence $\dvec$ is $(8,\alpha)$-stable. 
Write $M=M(\dvec)$ and $\Delta=\Delta(\dvec)$.
Then the switch chain on $\Omega(\dvec)$ mixes in polynomial time, with
mixing time $\tau(\varepsilon)$ which satisfies
\[
\tau(\varepsilon) \leq 12\,  \Delta^{14}\,n^{6}\, M^{3+\alpha} \, 
   \Big( \nfrac{1}{2}M\log M + \log(\varepsilon^{-1})\Big).
\]
\end{theorem}

Our next theorem gives a sufficient condition for a degree sequence to be $(8,8)$-stable, and a weaker condition which implies both P-stability and
strong stability.

\begin{theorem}\label{thm:stable}
\begin{enumerate}
\item[\emph{(a)}] Let $\dvec$ be a graphical degree sequence which satisfies
\[
\text{\Condone:} \quad  M(\dvec)>2 J(\dvec)+18\Delta(\dvec)+56.\qquad
\]
Then $\dvec$ is $(8,8)$-stable.
\item[\emph{(b)}] Suppose that ${\cal D}$ is a family of degree sequences such that 
every $\dvec\in {\cal D}$ satisfies
\[
\text{\Condtwo:} \quad M(\dvec)>2 J(\dvec)+6\Delta(\dvec)+2.\qquad
\]
 Then ${\cal D}$ is both $P$-stable, and strongly stable.  
\end{enumerate}
\end{theorem}

\noindent{\bf Remark.}\ We did not try to optimise the coefficient of $\Delta(\dvec)$ and the constant term in the assumptions of Theorem~\ref{thm:stable}. With more careful treatment in the proof, these numbers can be reduced.

As a further corollary we have the following mixing time bound on $\Omega(\dvec)$ when $\dvec$ satisfies the condition in Theorem~\ref{thm:stable}(a).

\begin{theorem}\label{thm:mixing2} Assume that ${\dvec}=(d_1,\ldots, d_n)$ is
a graphical degree sequence which satisfies \Condone.
Write $M=M(\dvec)$ and $\Delta = \Delta(\dvec)$.
Then the switch chain on $\Omega(\dvec)$ mixes in polynomial time, with 
mixing time $\tau(\varepsilon)$ which satisfies
\begin{equation}
\tau(\varepsilon) \leq 12\, \Delta^{14}\,n^{6}\, M^{11}\,  
   \Big(\nfrac{1}{2} M\log M + \log(\varepsilon^{-1})\Big). \label{mixingbound}
\end{equation}
\end{theorem}

\begin{proof} The result follows immediately from Theorems~\ref{thm:mixing} and~\ref{thm:stable}.
\end{proof}

\noindent {\bf Remark.}\
The mixing time bounds in Theorem~\ref{thm:mixing} and Theorem~\ref{thm:mixing2} are
probably far from tight, and are not the main focus of our work.  It is likely that
the mixing time bound that can be inferred from the results of Erd{\H o}s et al~\cite{hungarians}
is lower than our bound.  (It is harder to extract an explicit upper bound on the mixing
time of the switch chain from~\cite{AK}, so we do not make any comparison with their bound.)

\bigskip

The structure of the paper is as follows. After reviewing the necessary Markov
chain definitions and outlining the multicommodity flow argument from~\cite{GS},
we prove Theorem~\ref{thm:mixing} in Section~\ref{sec:flow} by employing a new 
counting argument. 
The proof of Theorem~\ref{thm:stable}(a) is given in Section~\ref{sec:stable}, 
while the proof of Theorem~\ref{thm:stable}(b) is deferred to Section~\ref{sec:comparison}.
In Section~\ref{sec:applications} we apply Theorem~\ref{thm:stable} to give
the first
proof of P-stability for several families of heavy-tailed degree sequences,
and provide an explicit upper bound for the mixing time of the switch chain for degree sequences
from these families.
In Section~\ref{sec:comparison}, $P$-stability and strong stability will be formally
defined and compared with our new notion of $k$-stability.  Then
the proof of Theorem~\ref{thm:stable}(b) will be presented.
Finally, some analogous results for the switch chain on directed graphs are
established in Section~\ref{sec:directed}.

\section{Polynomial mixing for 8-stable degree sequences}\label{sec:flow}

The aim of this section is the proof of Theorem~\ref{thm:mixing}.
First we must review some background.

\medskip

Let $\mathcal{M}$ be a reversible Markov chain with finite state space $\Omega$,
transition  matrix $P$ and stationary distribution $\pi$.  The
\emph{total variation distance} between two probability distributions
$\sigma$, $\sigma'$ on $\Omega$ is
\[ d_{\mathrm{TV}}(\sigma,\sigma') = \nfrac{1}{2} \sum_{x\in\Omega}
  |\sigma(x) - \sigma'(x)|.
\]
The \emph{mixing time} $\tau(\varepsilon)$ is defined by
\[ \tau(\varepsilon) = \max_{x\in\Omega}\,
           \min\{ T \geq 0 \mid   d_{\mathrm{TV}}(P^t_x,\pi)\leq \varepsilon
  \,\, \text{ for all $t \geq T$}\}
\]
where $P^t_x$ is the distribution of the state $X_t$ of $\mathcal{M}$ after
$t$ steps from the initial state $X_0=x$.
We say that the Markov chain $\mathcal{M}$ is \emph{rapidly mixing},
or has  \emph{polynomial mixing time}, if $\tau(\varepsilon)$ is
bounded from  above by some polynomial in $\log(|\Omega|)$ and $\log(\varepsilon^{-1})$.

Sinclair~\cite{sinclair} introduced multicommodity flow as a generalisation
of the canonical path method.
Let $\mathcal{G}$ be the directed graph underlying a Markov
chain $\mathcal{M}$, so that $xy$ is an edge of $\mathcal{G}$ if and only
if $P(x,y)>0$. A \emph{flow} in $\mathcal{G}$ is a function 
$f:\mathcal{P}\rightarrow [0,\infty)$ such that
\[ \sum_{p\in\mathcal{P}_{xy}} f(p) = \pi(x)\pi(y) \quad \text{ for all }
 \,\, x,y\in\Omega,\,\, x\neq y.\]
Here we define $\mathcal{P}_{xy}$ to be 
the set of all simple directed paths from $x$ to
$y$ in $\mathcal{G}$.  In practice, when defining a multicommodity flow,
we will set $f(p)=0$ for most paths, and focus on a much smaller set
of directed paths which will each carry a positive amount of flow.
Let $\mathcal{P} = \cup_{x\neq y} \mathcal{P}_{xy}$.
Extend $f$ to a function on oriented edges by setting
$f(e) = \sum_{p\ni e} f(p)$, so that $f(e)$ is the total flow routed through $e$.  
Write $Q(e) = \pi(x) P(x,y) =\pi(y) P(y,x)$
for the edge $e=xy$. The quantity $Q(e)$ is well-defined, by the reversibility of $\cal M$. Let $\ell(f)$ be the \textsl{length} of the longest path with
$f(p) > 0$, and let $\rho(e) = f(e)/Q(e)$ be the \emph{load} of the edge $e$.
The \emph{maximum load} of the flow is
$\rho(f) = \max_e \rho(e)$.
Using Sinclair~\cite[Proposition 1 and Corollary 6']{sinclair},  the mixing time of
$\mathcal{M}$ can be bounded above by
\begin{equation}
\label{flowbound} 
\tau(\varepsilon) \leq \rho(f)\ell(f)\left(\log(1/\pi^*) + \log(\varepsilon^{-1})\right)
\end{equation}
where $\pi^* = \min\{ \pi(x) \mid x\in\Omega\}$.

\subsection{The switch chain and multicommodity flow}\label{s:switch-chain}

The switch Markov chain $\mathcal{M}(\dvec)$ (or \emph{switch chain})
has state space
$\Omega(\dvec)$ and transitions given by the following procedure:
from the current state $G\in\Omegau$, choose an unordered pair of 
 distinct non-adjacent edges uniformly at random, 
say $F=\{\{ x,y\},\{ z,w\}\}$, and choose a perfect matching $F'$
from the set of three perfect matchings of (the complete graph on)
$\{ x,y,z,w\}$, chosen uniformly at random.  If 
$F'\cap \left( E(G) \setminus F\right) = \emptyset$ then the next state
is the graph $G'$ with edge set $\left(E(G) \setminus F\right)\cup F'$,
otherwise the next state is $G'=G$.

Define $M_2(\dvec) = \sum_{j=1}^n d_j(d_j-1)$. 
If $P(G,G')\neq 0$ and $G\neq G'$ then $P(G,G') = 1/\big(3 a(\dvec)\big)$, where
\begin{equation}
\label{ad}
 a(\dvec) = \binom{M(\dvec)/2}{2} - \dfrac{1}{2}\, M_2(\dvec)
\end{equation}
is the number of unordered pairs of distinct nonadjacent edges in $G$.
This shows that the switch chain $\mathcal{M}(\dvec)$ is symmetric,
and it is aperiodic since $P(G,G)\geq 1/3$ for all
$G\in\Omega(\dvec)$.

Cooper, Dyer and the second author~\cite{CDG} defined and analysed a 
multicommodity flow for the switch chain for regular degree sequences.
They proved that the switch chain has polynomial mixing time for
$\Omega(\dvec)$ whenever $\dvec = (d,d,\ldots, d)$ is a $d$-regular
sequence, for any $d=d(n)$.  The second author and Sfragara~\cite{GS}
showed how the analysis could be extended to irregular degree sequences
which were not too dense.  Surprisingly, as the proof given in~\cite{CDG}
is quite long and technical, there was only one lemma which relied on
regularity in its proof.  In~\cite{GS} a new argument was provided for this
``critical lemma'', leading to the extended rapid mixing result for irregular degree sequences.

To understand the purpose of the critical lemma, we provide a brief outline
of the multicommodity flow argument from~\cite{CDG,GS}.  The analysis
of the multicommodity flow is discussed in Section~\ref{ss:new-counting},
and the purpose of the critical lemma is given in (\ref{critical}).
Then in Section~\ref{sec:stable} we will give a new counting proof (Lemma~\ref{main}) 
which establishes the critical lemma when $\dvec$ is 8-stable. 

\bigskip

Let $G, G'\in\Omega(\dvec)$ be two graphs and let $G\triangle G'$ be
the symmetric difference of $G$ and $G'$, treated as a 2-coloured graph
(with edges from $G\setminus G'$ coloured blue and edges from $G'\setminus G$ 
coloured red, say). 
We define a set of directed paths
from $G$ to $G'$, and assign a value $f(p)$ to each of these paths, so that
$f$ is a flow.
\begin{itemize}
\item Define a bijection from the set of blue edges incident at $v$ to the
set of red edges incident at $v$, for each vertex $v\in \{ 1,\ldots, n\}$.
The vector of these bijections is called a \emph{pairing} $\psi$,
and the set of all possible pairings is denoted $\Psi(G,G')$.
\item  The pairing gives a canonical way to decompose the symmetric difference
$G\triangle G'$ into a sequence of simpler closed alternating walks,
called 1-\emph{circuits} and 2-\emph{circuits}.
\item  Each 1-circuit or 2-circuit is processed in a canonical way,
in order, to
give a segment of the canonical path $\gamma_\psi(G,G')$.
\end{itemize}
Thus, for each $(G,G')\in\Omega(\dvec)^2$ and each $\psi\in\Psi(G,G')$,
we define a (canonical) path $\gamma_{\psi}(G,G')$ from $G$ to $G'$.
For full details see~\cite[Section 2.1]{CDG}.

Next, the value of the flow along this path is defined as follows:
\begin{equation} f(\gamma_{\psi(G,G')}) = \frac{1}{|\Omega(\dvec)|^2\, |\Psi(G,G')|}
\label{f-def}
\end{equation}
and $f(p)=0$ for any other directed path from $G$ to $G'$.
Recall that $\mathcal{P}_{G,G'}$ is defined to be the set of all
directed paths from $G$ to $G'$, in the underlying digraph of the switch chain.
Summing $f(p)$ over all $p\in\mathcal{P}_{G,G'}$ gives
$1/|\Omega(\dvec)|^2 = \pi(G)\pi(G')$, as required for a valid flow.
This flow from $G$ to $G'$ has been equally shared among all paths in
$\{\gamma_{\psi}(G,G')\mid \psi\in\Psi(G,G')\}$.

\subsection{Encodings and the critical lemma}\label{ss:new-counting}

\begin{definition}\label{def:encoding}
An encoding $L$ of a graph $Z\in \Omega(\dvec)$ is an edge-labelled graph on
$n$ vertices, with edge labels in $\{ -1,1,2\}$,
such that
\begin{itemize}
\item[\emph{(i)}] the sum of edge-labels around vertex $j$ equals $d_j$ for all $j\in [n]$,
\item[\emph{(ii)}]  the edges with labels $-1$ or $2$ form a subgraph of one of the $10$~graphs
shown in Figure~\ref{f:zoo}.
(In the figure, ``\emph{?}'' represents a label which may be either $-1$ or $2$.)
\end{itemize}
An edge labelled $-1$ or $2$ is a \emph{defect edge}.  
\end{definition}

In the analysis of the multicommodity flow in~\cite{CDG,GS}, encodings play
an important role.  Given $G,G',Z\in\Omega(\dvec)$, 
identify each of $Z, G, G'$ with their symmetric 0-1 adjacency matrix
and define the matrix $L$ by  
\[ L + Z = G + G'.\]
Then $L$ corresponds to an edge-labelled graph,
which we also denote by $L$. 
It follows from the next result, proved in~\cite{GS}, that the edge-labelled graph $L$ 
satisfies the definition of encoding given above. We call $L$
the encoding of $Z$ with respect to $(G,G')$.  

\begin{lemma}[{\cite[Lemma 2.1(ii)]{GS}}]
\
Let $(G,G')\in\Omega(\dvec)^2$ and $\psi\in \Psi(G,G')$, and suppose that 
$(Z,Z')$ is
a transition of the switch chain which forms part of the canonical
path $\gamma_{\psi}(G,G')$.
Then the encoding $L$ of $Z$ with respect to $(G,G')$
has at most four defect edges, which must form a subgraph of one of the
10 possible edge-labelled graphs shown in Figure~\ref{f:zoo}.
\end{lemma}

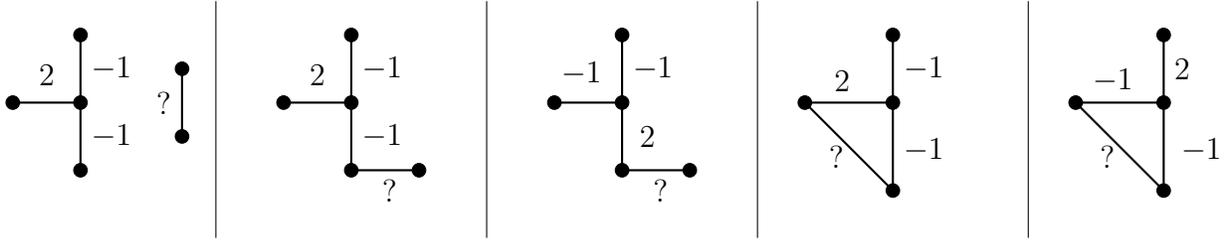
\begin{figure}[ht!]
\begin{center}
\begin{tikzpicture}[scale=0.9]
\draw [-,thick] (1.0,3) -- (1.0,5);
\draw [-,thick] (-0.0,4) -- (1.0,4);
\draw [-,thick]  (2.5,3.5) -- (2.5,4.5);
\draw [fill] (1.0,3) circle (0.1); \draw [fill] (1.0,5) circle (0.1);
\draw [fill] (1.0,4) circle (0.1); \draw [fill] (0.0,4) circle (0.1);
\draw [fill] (2.5,3.5) circle (0.1); \draw [fill] (2.5,4.5) circle (0.1);
\node [above] at (0.5,4.1) {$2$};
\node [right] at (1.0,4.5) {$-1$};
\node [right] at (1.0,3.5) {$-1$};
\node [left] at (2.5,4.0) {$?$};
%%%%%%%%%%%%%
\draw [-,thick] (4,4) -- (5,4);
\draw [-,thick] (5,5) -- (5,3) -- (6,3);
\draw [fill] (4,4) circle (0.1); \draw [fill] (5,4) circle (0.1);
\draw [fill] (5,5) circle (0.1); \draw [fill] (5,3) circle (0.1);
\draw [fill] (6,3) circle (0.1); 
\node [above] at (4.5,4.1) {$2$};
\node [right] at (5.0,4.5) {$-1$};
\node [right] at (5.0,3.5) {$-1$};
\node [right] at (5.3,2.7) {$?$};
%%%%%%%%%%%%%
\draw [-,thick] (8.0,4) -- (9.0,4);
\draw [-,thick] (9.0,5) -- (9.0,3) -- (10.0,3);
\draw [fill] (8.0,4) circle (0.1); \draw [fill] (9.0,4) circle (0.1);
\draw [fill] (9.0,5) circle (0.1); \draw [fill] (9.0,3) circle (0.1);
\draw [fill] (10.0,3) circle (0.1); 
\node [above] at (8.4,4.1) {$-1$};
\node [right] at (9.0,4.5) {$-1$};
\node [right] at (9.1,3.5) {$2$};
\node [right] at (9.3,2.7) {$?$};
%%%%%%%%%%%%%
\begin{scope}[shift={(10,3.5)}]
\draw [-,thick] (3,1.5) -- (3,-0.8) -- (1.7,0.5) -- (3,0.5);
\draw [fill] (3,1.5) circle (0.1); \draw [fill] (3,-0.8) circle (0.1);
\draw [fill] (1.7,0.5) circle (0.1); \draw [fill] (3,0.5) circle (0.1);
\node [above] at (2.25,0.5) {$2$};
\node [right] at (3.0,-0.2) {$-1$};
\node [right] at (3.0,1.0) {$-1$};
\node [right] at (1.9,-0.3) {$?$};
\end{scope}
%%%%%%%%%%%%%
\begin{scope}[shift={(10,3.5)}]
\draw [-,thick] (7,1.5) -- (7,-0.8) -- (5.7,0.5) -- (7,0.5);
\draw [fill] (7,1.5) circle (0.1); \draw [fill] (7,-0.8) circle (0.1);
\draw [fill] (5.7,0.5) circle (0.1); \draw [fill] (7,0.5) circle (0.1);
\node [above] at (6.25,0.5) {$-1$};
\node [right] at (7.1,-0.2) {$-1$};
\node [right] at (7.0,1.0) {$2$};
\node [right] at (5.9,-0.3) {$?$};
\end{scope}
%%%%%%%%%%%%%
\draw [-] (3,2) -- (3,5.5);
\draw [-] (7,2) -- (7,5.5);
\draw [-] (11,2) -- (11,5.5);
\draw [-] (15,2) -- (15,5.5);
\end{tikzpicture}
\caption{The defect edges in $L$ form a subgraph of one of these graphs,
as proved in~\cite[Lemma~2.1(ii)]{GS}. 
The ``?'' stands for a defect edge which may be labelled $-1$ or 2.}
\label{f:zoo}
\end{center}
\end{figure}

Note that this result holds for any degree sequence $\dvec$, as the
restrictions on $\dvec$ which are needed for the rapid mixing result
in~\cite{GS} only arise in the proof of the ``critical lemma''.

Given $Z\in \Omega(\dvec)$, say that encoding $L$ is \emph{consistent with $Z$} if
$L+Z$ only takes entries in $\{0,1,2\}$ (again identifying $L$ and $Z$ with their
adjacency matrices).  Equivalently, $L$ is consistent with $Z$ if any edge
which has label $-1$ in $L$ must also be an edge of $Z$.
Let $\mathcal{L}^\dagger(Z)$ be the set of encodings which are consistent with $Z$.
The task of the critical lemma is to prove that
\begin{equation}
\label{critical}
 |\mathcal{L}^\dagger(Z)|\,\,  \text{is at most polynomially larger than}
 \,\, |\Omega(\dvec)|
\end{equation}
for all $Z\in\Omega(\dvec)$ and for $\dvec$ which satisfy some condition
required to make the proof work.  (The set of encodings considered
in the critical lemma is slightly different in each of ~\cite{CDG,GS} and the
present work, but this does not matter as long as the set contains all
encodings $L$ which arise along a canonical path (that is, encodings
corresponding to some state $Z$ which belongs to a path $\gamma_{\psi}(G,G')$).
)

Arguing as in~\cite[Theorem~1.1]{GS}, for
example, it can be shown that the switch chain is rapidly mixing on
$\Omega(\dvec)$ for any $\dvec$ which satisfies the condition of the
critical lemma.  We encapsulate this argument into the following result,
which can then be used as a black box.

\begin{theorem}\label{thm:catherine}
Let $\dvec$ be a graphical degree sequence. 
 Write $M = M(\dvec)$ and $\Delta = \Delta(\dvec)$. 
Suppose that there exists 
a function $g(\dvec)$, which depends only on $\dvec$, such that
\[
 |\mathcal{L}^\dagger(Z)| \leq  g(\dvec)\, |\Omega(\dvec)|
\]
for all $Z\in\Omega(\dvec)$.
Then the mixing time $\tau(\varepsilon)$ of the switch chain on $\Omega(\dvec)$ 
satisfies
   \[
   \tau(\varepsilon) \leq \nfrac{1}{2} g(\dvec)\, \Delta^{14}\, M^3\,
   \Big( \nfrac{1}{2} M\log M + \log(\varepsilon^{-1})\Big).
   \]
\label{Lbound}
\end{theorem}

\begin{proof}
We follow the structure of the argument used to prove~\cite[Theorem~1.1]{GS}, 
working towards an application of (\ref{flowbound}).
As noted in~\cite[Equation (1)]{GS}, 
\[ |\Omega(\dvec)| \leq
          \exp\left( \nfrac{1}{2} \, M\log M \right).
\]
Hence, as the stationary distribution $\pi$ is uniform,
 the smallest stationary probability $\pi^\ast$ satisfies
\begin{equation}
\label{pi-min}
\log(1/\pi^\ast) = \log(|\Omega(\dvec)|) \leq \nfrac{1}{2} M\log M. 
\end{equation}
Next, $\ell(f)\leq M/2$ since each transition along a canonical path from $G$ to $G'$
replaces an edge of $G$ by an edge of $G'$.

Let $e=(Z,Z')$ be a transition of the switch chain. Then
\[ 1/Q(e) = 6\, a(\dvec)\, |\Omega(\dvec)| \leq M^2\, |\Omega(\dvec)|,\]
using (\ref{ad}).
(Note, the factor $|\Omega(\dvec)|$ was missing in the corresponding bound
in the proof of~\cite[Theorem~1.1]{GS}, so we correct that typographical error here.)
 
The argument in~\cite{GS} used a set of encodings $\mathcal{L}^\ast(Z)$ which 
is a proper subset of $\mathcal{L}^\dagger(Z)$.  Specifically, $\mathcal{L}^\ast(Z)$ is
the set of all encodings in $\mathcal{L}^\dagger(Z)$ which also satisfy the
conclusions of~\cite[Lemma~2.2]{GS}.  Then
\cite[Lemma~2.3]{GS} proves that for the flow $f$ and transition $e$,
\[  f(e) \leq \Delta^{14}\, \frac{|\mathcal{L}^\ast(Z)|}{|\Omega(\dvec)|^2}.
\]
(Again, we emphasise that this is still true for degree sequences that do not
satisfy the condition of the rapid mixing theorem from~\cite{GS},
since that condition was only used in the proof of the ``critical lemma''.)
We may replace $\mathcal{L}^\ast(Z)$ with $\mathcal{L}^\dagger(Z)$ to obtain
a weaker upper bound on $f(e)$, and hence
\[ \rho(f) = \max_e \frac{f(e)}{Q(e)} \leq \Delta^{14} M^2\, \max_{Z\in \Omega(\dvec)} \frac{|\mathcal{L}^\dagger(Z)|}{|\Omega(\dvec)|} \leq g(\dvec)\, \Delta^{14}\, M^2.\]
Substituting this bound, the bound $\ell(f)\leq M/2$ and (\ref{pi-min}) 
into (\ref{flowbound}) completes the proof.
\end{proof}

\subsection{Proof of Theorem~\ref{thm:mixing}}

Next we give a new counting proof, which establishes the ``critical lemma''
 when $\dvec$ is 8-stable.

\begin{lemma}\label{lem:counting}
Assume that the graphical degree sequence $\dvec$ is $(8,\alpha)$-stable
for some nonnegative real number $\alpha$. Then for any $Z\in\Omega(\dvec)$,
\[  |\mathcal{L}^\dagger(Z)|\le 24\, n^6 M(\dvec)^{\alpha}\, |\Omega(\dvec)|.\] 
\end{lemma}

\begin{proof}
We know that $|\mathcal{L}^\dagger(Z)|$ is bounded from above by the number of edge-labelled graphs 
satisfying conditions (i) and (ii) from Definition~\ref{def:encoding}.
First, we bound the number of ways to choose the defect edges in $Z$.
These defect edges must form a subgraph of one of the 10 edge-labelled graphs shown in
Figure~\ref{f:zoo}, recalling that ``?'' can be either $-1$ or 2. 
Each of these 10 graphs has 4 edges and at most 6 vertices.
If $H$ has $\ell$ vertices then the number of injections $\varphi:V(H)\to [n]$
is at most $n^\ell$. 
It follows that the number of ways to choose $H$ and $\varphi$
is at most
\[
10\Big( n^6 + 4 n^5 + 6 n^4 + 4 n^2  + 1\Big) 
  \leq 24 n^6
\]
since $n\geq 4$ (or no switch is possible).

Now let $\mathcal{E}$ be the set of chosen (labelled) defect edges.
We bound the number of ways to complete $\mathcal{E}$ to an encoding $L\in\mathcal{L}^{\dagger}(Z)$. 
All edges in $L\setminus \mathcal{E}$ are labelled 1. For all $j\in [n]$,
the number of edges in $L\setminus \mathcal{E}$ incident with vertex $j$ is equal to 
$d'_j=d_j - x_j$,
where $x_j$ is the sum of edge-labels from $\mathcal{E}$ around vertex $j$. 
Let $\dvec'=(d'_1,\ldots,d'_n)$ and observe that $\norm{\dvec'}_1$ is even. 
The number of encodings $L\in\mathcal{L}^{\dagger}(Z)$ 
such that the set of defect edges given 
by ${\mathcal E}$ is at most $|\Omega(\dvec')|$. 
Checking through all possible graphs $H$, we confirm that $\norm{\dvec-\dvec'}_1\le 8$ always. 
(As an example, let $H$ be the first option in Figure~\ref{f:zoo},
with ``?'' replaced by~2.  There are three vertices with $d'_j = d_j + 2$,
two with $d'_j = d_j-1$, and $d'_j=d_j$ for all other vertices.
Hence $\norm{\dvec-\dvec'}_1= 8$ in this case. The subgraph formed
from $H$ by deleting one of the $(-1)$-defect edges also satisfies
$\norm{\dvec-\dvec'}_1 = 8$.)
Since $\dvec$ is $(8,\alpha)$-stable, we have $|\Omega(\dvec')| \le M(\dvec)^{\alpha}\, |\Omega(\dvec)|$. 
The result follows as there are at most $24n^6$ ways to fix ${\cal E}$.
\end{proof}

\begin{proof}[Proof of Theorem~\ref{thm:mixing}.]
This follows by combining Theorem~\ref{thm:catherine} and Lemma~\ref{lem:counting}.
\end{proof}

\section{Stable degree sequences and proof of Theorem~\ref{thm:stable}}
\label{sec:stable}

Assume that $\dvec$ is graphical. For positive integers $k$, let 
\begin{equation}
\label{eq:Nk-def}
 \N_k(\dvec)=\{\dvec'\in {\mathbb N}^n:\ \norm{\dvec'}_1 \equiv 0\!\!\! \pmod 2,\quad \norm{\dvec'-\dvec}_1\le k\}.
\end{equation}

\begin{lemma}\label{lem:transitive}
Suppose that every graphical degree sequence $\dvec'\in \N_6(\dvec)$ is $(2,\alpha)$-stable. 
Then $\dvec$ is $(8,4\alpha)$-stable. 
\end{lemma}

\begin{proof}
Assume that $\dvec'$ is a degree sequence such that $\norm{\dvec-\dvec'}_1\leq 8$.
Then we can find $\dvec'=\dvec^{(8)}, \dvec^{(6)},\dvec^{(4)},\dvec^{(2)},\dvec^{(0)}=\dvec,$ such that $\dvec^{(i)}\in \N_i(\dvec)$ and $\norm{\dvec^{(i+1)}-\dvec^{(i)}}_1\leq 2$ for every $i\in\{0,2,4,6\}$. 
By assumption, $|\Omega(\dvec^{(i+1)})|\le M^{\alpha}\, |\Omega(\dvec^{(i)})|$ for every $i\in\{0,2,4,6\}$. Consequently $|\Omega(\dvec')|\le M^{4\alpha}|\Omega(\dvec)|$, and the assertion follows.
\end{proof}

A directed 2-path in a graph $G$ is an ordered triple of distinct vertices $(a,b,c)$ such that $ab$ and $bc$ are edges in $G$.
For any graph $G$ with degree sequence $\dvec$, and any $v\in G$, the number of directed 2-paths which start at $v$ (that is, with $v=a$) is at most 
\begin{equation}
\sum_{i=1}^{d_v} (d_{\mu(i)}-1)  \leq
\sum_{i=1}^{d_{\mu(1)}} (d_{\mu(i)}-1) 
\le J(\dvec)-\Delta(\dvec). \label{2paths}
\end{equation}

The key result of this section is the following, which will be proved in Section~\ref{s:deferred}.

\begin{theorem}\label{thm:main}
If $M(\dvec)>2 J(\dvec)+6\Delta(\dvec)+2$  then $\dvec$ is $(2,2)$-stable.
\end{theorem}

\medskip

Using Theorem~\ref{thm:main}, we can now prove Theorem~\ref{thm:stable}(a).
(The proof of Theorem~\ref{thm:stable}(b) is deferred to the end of
Section~\ref{sec:comparison}.)

\begin{proof}[Proof of Theorem~\ref{thm:stable}(a)]
First suppose that $\dvec$ satisfies \Condone\ and let 
$\dvec'\in\mathcal{N}_6(\dvec)$.
Then $M(\dvec')\ge M(\dvec)-6$, and $\Delta(\dvec')\le \Delta(\dvec)+6$. 
Also $J(\dvec')\le J(\dvec)+6\Delta(\dvec)+6$, with equality when 
$\Delta(\dvec') = \Delta(\dvec)+6$ and the largest $\Delta(\dvec')$ entries of 
$\dvec$ all equal $\Delta(\dvec)$.
Applying Theorem~\ref{thm:main} to $\dvec'$ shows that $\dvec'$ is
(2,2)-stable, and therefore $\dvec$ is (8,8)-stable, 
by Lemma~\ref{lem:transitive}.  This establishes part~(a) of Theorem~\ref{thm:stable}.
\end{proof}

We conclude this section by proving Theorem~\ref{thm:main}. 

\subsection{ Proof of Theorem~\ref{thm:main}}\label{s:deferred}

The switching method is a way to bound or approximate the ratio of
the sizes of large, finite sets, introduced by McKay~\cite{mckay-line}.  
It has been used to obtain asymptotic enumeration formulae for
sparse graphs with given degrees, see~\cite{McKW91}, and also forms
the basis for fast exactly-uniform sampling algorithms, see~\cite{GW,McKW90}.
Our proof of Theorem~\ref{thm:main} uses three simple switching arguments.

\begin{proof}[Proof of Theorem~\ref{thm:main}]
For all $t\in [n]$
we write $\boldsymbol{e}_t$ to denote the vector with~1 in the $t$'th position and zeroes elsewhere.
 Let $\dvec'\in \N_2(\dvec)$ such that $\norm{\dvec-\dvec'}_1=2$. We will prove that $|\Omega(\dvec')|\le M(\dvec)^2|\, \Omega(\dvec)|$, which implies that $\dvec$ is $(2,2)$-stable, as claimed. 

We will consider three cases for $\dvec'$: 
\[ \dvec'=\dvec-\boldsymbol{e}_i-\boldsymbol{e}_j,\qquad
   \dvec'=\dvec+\boldsymbol{e}_i+\boldsymbol{e}_j, \qquad
   \dvec'=\dvec+\boldsymbol{e}_i-\boldsymbol{e}_j.
\]
Write $u\sim v$ to indicate that vertices $u$ and $v$ are adjacent.

\bigskip

\noindent {\em Case 1: $\dvec'=\dvec+\boldsymbol{e}_i+\boldsymbol{e}_j$.}\ 
There are two subcases: $i\neq j$ and $i=j$.  
\medskip

\noindent {\em Case 1a: $i\neq j$.} We will use the $(i^-,j^-)$-switching shown on the left of Figure~\ref{f:1a1b}.
This switching converts a graph $G'\in \Omega(\dvec')$ to a graph $G\in \Omega(\dvec)$. To perform the switching, choose an ordered set of vertices $(u_1,u_2,u_3,u_4)$ in $G'$ such that 
\begin{enumerate}
\item[(a)] the four vertices are distinct from $i,j$ and are pairwise distinct except that $u_1=u_4$ is permitted;
\item[(b)] $i\sim u_1$, $u_2\sim u_3$, $j\sim u_4$;
\item[(c)] $u_1$ is not adjacent with $u_2$ and $u_3$ is not adjacent with $u_4$.
\end{enumerate}
Then the switching deletes edges $iu_1$, $ju_4$ and $u_2u_3$, and add edges $u_1u_2$ and $u_3u_4$. The resulting graph $G$ has degree sequence $\dvec$.

\begin{figure}[ht!]
\begin{center}
\begin{tikzpicture}[scale=0.8]
\draw [-,very thick] (0,0) -- (0,2) (0,4) -- (2,4) (2,2) -- (2,0);
\draw [-,very thick,dashed] (0,2) -- (0,4) (2,2) -- (2,4);
\draw [fill] (0,0) circle (0.12); \draw [fill] (0,2) circle (0.12);
\draw [fill] (0,4) circle (0.12); \draw [fill] (2,0) circle (0.12);
\draw [fill] (2,2) circle (0.12); \draw [fill] (2,4) circle (0.12);
\node [left] at (-0.2,0) {$i$}; \node [right] at (2.2,0) {$j$};
\node [left] at (-0.1,2) {$u_1$}; \node [right] at (2.1,2) {$u_4$};
\node [left] at (-0.1,4) {$u_2$}; \node [right] at (2.1,4) {$u_3$};
\draw [->,line width = 1mm] (3.2,2) -- (4.0,2);
\begin{scope}[shift={(5.25,0)}]
\draw [-,very thick,dashed] (0,0) -- (0,2) (0,4) -- (2,4) (2,2) -- (2,0);
\draw [-,very thick] (0,2) -- (0,4) (2,2) -- (2,4);
\draw [fill] (0,0) circle (0.12); \draw [fill] (0,2) circle (0.12);
\draw [fill] (0,4) circle (0.12); \draw [fill] (2,0) circle (0.12);
\draw [fill] (2,2) circle (0.12); \draw [fill] (2,4) circle (0.12);
\node [left] at (-0.2,0) {$i$}; \node [right] at (2.2,0) {$j$};
\node [left] at (-0.1,2) {$u_1$}; \node [right] at (2.1,2) {$u_4$};
\node [left] at (-0.1,4) {$u_2$}; \node [right] at (2.1,4) {$u_3$};
\end{scope}
%%%%%%
\draw [-] (9.25,-0.5) -- (9.25,4.5); % divider
%%%%%%
\begin{scope}[shift={(11,0)}]
\draw [-,very thick] (2,2) -- (1,0) -- (0,2) (0,4) -- (2,4); 
\draw [-,very thick,dashed] (0,2) -- (0,4) (2,2) -- (2,4);
\draw [fill] (1,0) circle (0.12); \draw [fill] (0,2) circle (0.12);
\draw [fill] (0,4) circle (0.12);
\draw [fill] (2,2) circle (0.12); \draw [fill] (2,4) circle (0.12);
\node [left] at (0.8,0) {$i$}; 
\node [left] at (-0.1,2) {$u_1$}; \node [right] at (2.1,2) {$u_4$};
\node [left] at (-0.1,4) {$u_2$}; \node [right] at (2.1,4) {$u_3$};
\draw [->,line width = 1mm] (3.2,2) -- (4.0,2);
\begin{scope}[shift={(5.25,0)}]
\draw [-,very thick,dashed] (2,2) -- (1,0) -- (0,2) (0,4) -- (2,4); 
\draw [-,very thick] (0,2) -- (0,4) (2,2) -- (2,4);
\draw [fill] (1,0) circle (0.12); \draw [fill] (0,2) circle (0.12);
\draw [fill] (0,4) circle (0.12); 
\draw [fill] (2,2) circle (0.12); \draw [fill] (2,4) circle (0.12);
\node [left] at (0.8,0) {$i$}; 
\node [left] at (-0.1,2) {$u_1$}; \node [right] at (2.1,2) {$u_4$};
\node [left] at (-0.1,4) {$u_2$}; \node [right] at (2.1,4) {$u_3$};
\end{scope}
\end{scope}
\end{tikzpicture}
\caption{The $(i^-,j^-)$-switching.  
Left: Case 1a with $i\neq j$.\,\,\,  
Right: Case 1b  with $i=j$.}
\label{f:1a1b}
\end{center}
\end{figure}
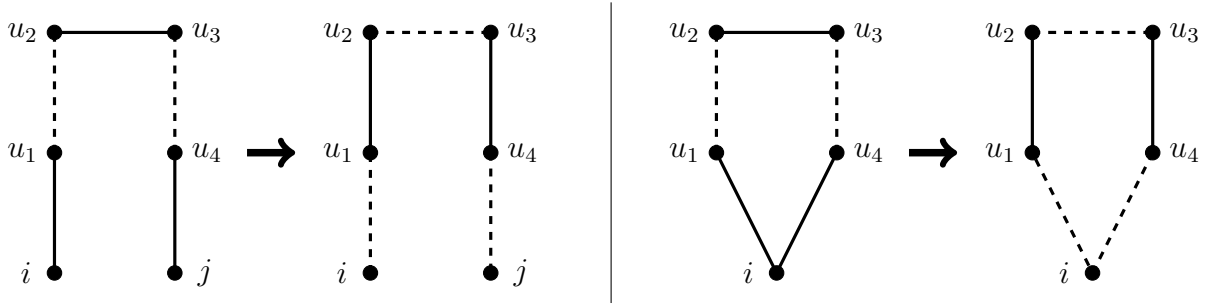

Let $f_{(i^-,j^-)}(G')$ denote the number of ways to perform an $(i^-,j^-)$-switching to $G'$. As $G'\in \Omega(\dvec')$, there are $d_i+1$ ways to choose $u_1$, $d_j+1$ ways to choose $u_4$, and at most $M(\dvec')=M(\dvec)+2$ ways to choose $(u_2,u_3)$. Hence $f_{(i^-,j^-)}(G)\le (d_i+1)(d_j+1)(M(\dvec)+2)$. Among all such choices, if condition (a) above fails, we say there is a vertex collision. To obtain a lower bound,
we use inclusion-exclusion and need to exclude the choices where $\{u_2,u_3\}\cap \{i,j,u_1,u_4\}\neq \emptyset$, and the choices where $u_1\sim u_2$ or $u_3\sim u_4$.  Given $u_1,u_4$, the number of choices for  $(u_2,u_3)$ such that $\{u_2,u_3\}\cap \{i,j,u_1,u_4\}\neq \emptyset$ is at most $8\Delta(\dvec)+4$, noting that $d'_i,d'_j\le \Delta(\dvec)+1$.  Next, we bound the number of choices for  $(u_2,u_3)$ where there is no vertex collision but  $u_1\sim u_2$ or $u_3\sim u_4$. 
Since $\dvec'$ agrees with $\dvec$ in all components except for $i$ and $j$, 
adapting (\ref{2paths}) shows that the number of 2-paths in $G'$ 
which start at $u_1$ and avoid both $i$ and $j$ is at most 
\[
\sum_{\ell=1}^{d_{u_1}-1} (d_{\pi(\ell)}-1) \le \sum_{\ell=1}^{\Delta(\dvec)-1} (d_{\pi(\ell)}-1) \le J(\dvec)-\Delta(\dvec).
\]
Hence, given $u_1$ and $u_4$, the number of choices for  $(u_2,u_3)$ where there is no vertex collision but $u_1\sim u_2$ (or similarly $u_3\sim u_4$) is at most $J(\dvec)-\Delta(\dvec)$. It follows that, for every $G'\in \Omega(\dvec')$,
\begin{align*}
 f_{(i^-,j^-)}(G')&\ge (d_i+1)(d_j+1)(M(\dvec)+2-8\Delta(\dvec)-4- 2(J(\dvec)-\Delta(\dvec))\\
 &=(d_i+2)(d_i+1)(M(\dvec)- 2J(\dvec)-6\Delta(\dvec)-2 ),
\end{align*}
which is positive by the assumption of the theorem.

Next, given $G\in\Omega(\dvec)$, we estimate $b_{(i^-,j^-)}(G)$,  the number of ways to create $G$ by performing an $(i^-,j^-)$-switching to a graph in $\Omega(\dvec')$. To estimate $b_{(i^-,j^-)}(G)$ we count the number of inverse $(i^-,j^-)$-switchings that can be performed to $G$. An inverse $(i^-,j^-)$-switching chooses vertices $(u_1,u_2,u_3,u_4)$ such that
\begin{enumerate}
\item[(a)] the four vertices are distinct from $i,j$ and are pairwise distinct except that $u_1=u_4$ is permitted;
\item[(b)] $iu_1$, $u_2u_3$ and $ju_4$ are not edges;
\item[(c)] $u_1\sim u_2$ and $u_3\sim u_4$.
\end{enumerate}
Trivially we have $b_{(i^-,j^-)}(G')\le M(\dvec)^2$.
Hence
\[
\frac{|\Omega(\dvec')|}{|\Omega(\dvec)|} \le \frac{\max_{G\in \Omega(\dvec)} b_{(i^-,j^-)}(G)}{\min_{G'\in \Omega(\dvec')}f_{(i^-,j^-)}(G')} \le \frac{M(\dvec)^2}{(d_i+1)(d_j+1)(M(\dvec)- 2J(\dvec)-6\Delta(\dvec) -2)}\le M(\dvec)^2.
\]

\medskip

\noindent {\em Case 1b: $i=j$.} We use the $(2i^-)$-switching, 
which is defined the same as the $(i^-,j^-)$-switching except that $i=j$,
and now we require that $u_1\neq u_4$.  
This switching is shown on the right of Figure~\ref{f:1a1b}.
Trivially, we have $f_{2i^-}(G')\le (d_i+2)(d_i+1)(M(\dvec)+2)$. Given $u_1$ and $u_4$, the number of choices for $(u_2,u_3)$ such that  $\{u_2,u_3\}\cap \{i,u_1,u_4\}\neq \emptyset$ is at most $2\cdot(3\Delta(\dvec)+2)=6\Delta(\dvec)+4$, noting that $d'_i\le \Delta(\dvec)+2$. The number of choices where there is no vertex collision but $u_1\sim u_2$ or $u_3\sim u_4$ is at most $2(J(\dvec)-\Delta(\dvec))$, as shown in Case~1a. 
(Again, this uses the fact that $\dvec'$ agrees with $\dvec$
everywhere except $i$ and $j$.)
Therefore
\begin{align*}
f_{2i^-}(G')&\ge (d_i+2)(d_i+1)(M(\dvec)+2-(6\Delta(\dvec)+4)- 2(J(\dvec)-\Delta(\dvec))) \\
&=(d_i+2)(d_i+1)(M(\dvec)- 2J(\dvec)-4\Delta(\dvec)-2),
\end{align*}
which is positive by the theorem assumption.
We also have the trivial upper bound $b_{2i^-}(G)\le M(\dvec)^2$.
It follows that
\[
\frac{|\Omega(\dvec')|}{|\Omega(\dvec)|} \le \frac{\max_{G\in \Omega(\dvec)} b_{(i^-,j^-)}(G)}{\min_{G'\in \Omega(\dvec')}f_{(i^-,j^-)}(G')} \le \frac{M(\dvec)^2}{(d_i+2)(d_i+1)(M(\dvec)- 2J(\dvec)-4\Delta(\dvec)-2 )}\le M(\dvec)^2.
\]

\bigskip

\noindent {\em Case 2: $\dvec'=\dvec-\boldsymbol{e}_i-\boldsymbol{e}_j$.}\, 
Again there are two subcases:\ $i\neq j$ and $i=j$.\medskip

\noindent {\em Case 2a: $i\neq j$.}
We will use the $(i^+,j^+)$-switching. 
This switching chooses an ordered set of vertices $(u_1,u_2)$ in $G'\in \Omega(\dvec')$ such that 
\begin{enumerate}
\item[(a)] the four vertices $u_1,u_2,i,j$ are pairwise distinct;
\item[(b)] $u_1\sim u_2$;
\item[(c)] $u_1$ is not adjacent with $i$ and $u_2$ is not adjacent with $j$.
\end{enumerate}
The switching deletes the edge $u_1u_2$ and adds edges $iu_1$ and $ju_2$. The result is a graph $G\in\Omega(\dvec)$.
See the left hand side of Figure~\ref{f:2a2b}.

\begin{figure}[ht!]
\begin{center}
\begin{tikzpicture}[scale=0.8]
\draw [-,very thick,dashed] (0,0) -- (0,2) (2,2) -- (2,0);
\draw [-,very thick] (0,2) -- (2,2);
\draw [fill] (0,0) circle (0.12); \draw [fill] (0,2) circle (0.12);
\draw [fill] (2,0) circle (0.12); \draw [fill] (2,2) circle (0.12);
\node [left] at (-0.2,0) {$i$}; \node [right] at (2.2,0) {$j$};
\node [left] at (-0.1,2) {$u_1$}; \node [right] at (2.1,2) {$u_2$};
\draw [->,line width = 1mm] (3.2,1) -- (4.0,1);
\begin{scope}[shift={(5.25,0)}]
\draw [-,very thick] (0,0) -- (0,2) (2,2) -- (2,0);
\draw [-,very thick,dashed] (0,2) -- (2,2);
\draw [fill] (0,0) circle (0.12); \draw [fill] (0,2) circle (0.12);
\draw [fill] (2,0) circle (0.12); \draw [fill] (2,2) circle (0.12);
\node [left] at (-0.2,0) {$i$}; \node [right] at (2.2,0) {$j$};
\node [left] at (-0.1,2) {$u_1$}; \node [right] at (2.1,2) {$u_2$};
\end{scope}
%%%%%%
\draw [-] (9.25,-0.5) -- (9.25,2.5); 
%%%%%%
\begin{scope}[shift={(11,0)}]
\draw [-,very thick,dashed] (2,2) -- (1,0) -- (0,2);
\draw [-,very thick] (0,2) -- (2,2);
\draw [fill] (1,0) circle (0.12); \draw [fill] (0,2) circle (0.12);
\draw [fill] (2,2) circle (0.12); 
\node [left] at (0.8,0) {$i$}; 
\node [left] at (-0.1,2) {$u_1$}; \node [right] at (2.1,2) {$u_2$};
\draw [->,line width = 1mm] (3.2,1) -- (4.0,1);
\begin{scope}[shift={(5.25,0)}]
\draw [-,very thick] (2,2) -- (1,0) -- (0,2);
\draw [-,very thick,dashed] (0,2) -- (2,2);
\draw [fill] (1,0) circle (0.12); \draw [fill] (0,2) circle (0.12);
\draw [fill] (2,2) circle (0.12); 
\node [left] at (0.8,0) {$i$}; 
\node [left] at (-0.1,2) {$u_1$}; \node [right] at (2.1,2) {$u_2$};
\end{scope}
\end{scope}
\end{tikzpicture}
\caption{The $(i^+,j^+)$-switching.  Left: Case 2a with $i\neq j$.\,\,\,  Right, Case 2b with $i=j$.}
\label{f:2a2b}
\end{center}
\end{figure}
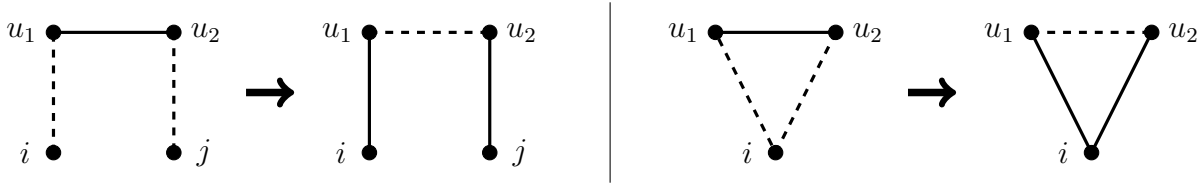

Obviously, $f_{(i^+,j^+)}(G)\le M(\dvec')=M(\dvec)-2$. To apply inclusion-exclusion, we need to subtract the number of choices where (a) or (c) is violated. 
The number of choices where a vertex collision occurs (that is, where (a) is violated) is at most $4\Delta(\dvec')\le 4\Delta(\dvec)$. 
The number of choices without vertex collision but where $i\sim u_1$ (or $j\sim u_2$) 
is at most $J(\dvec)-\Delta(\dvec)$, as shown in Case~1a.
Hence,
\[
f_{(i^+,j^+)}(G')\ge M(\dvec)-2-4\Delta(\dvec)-2(J(\dvec)-\Delta(\dvec)) =M(\dvec)-2J(\dvec)-2\Delta(\dvec)-2,
\]
which is positive by the assumption of the theorem. We also have the trivial upper bound 
\[
b_{(i^+,j^+)}(G)\le d_i d_j.
\]
Hence,
\[
\frac{|\Omega(\dvec')|}{|\Omega(\dvec)|} \le \frac{\max_{G\in \Omega(\dvec)} b_{(i^+,j^+)}(G)}{\min_{G'\in \Omega(\dvec')}f_{(i^+,j^+)}(G')} \le \frac{d_id_j }{M(\dvec)-2J(\dvec)-2\Delta(\dvec)-2}\le M(\dvec)^2.
\]

\medskip

\noindent {\em Case 2b: $i=j$.}
We will use the $2i^+$-switching, which is the same as the $(i^+,j^+)$-switching but with $i=j$.  See the right hand side of Figure~\ref{f:2a2b}.
Arguing as in Case 2a gives
\[
f_{2i^+}(G')\ge M(\dvec)-2-2\Delta(\dvec)-2(J(\dvec)-\Delta(\dvec)) =M(\dvec)-2J(\dvec)-2.
\]
Together with the trivial upper bound $b_{2i^+}(G)\le d_i (d_i-1)$
we have
\[
\frac{|\Omega(\dvec')|}{|\Omega(\dvec)|} \le \frac{\max_{G\in \Omega(\dvec)} b_{2i^+}(G)}{\min_{G'\in \Omega(\dvec')}f_{2i^+}(G')} \le \frac{d_i(d_i-1) }{M(\dvec)-2J(\dvec)+2}\le M(\dvec)^2.
\]

\bigskip

\noindent {\em Case 3: $\dvec'=\dvec+\boldsymbol{e}_i-\boldsymbol{e}_j$ and $i\neq j$.}
In this case we will use the $(i^-,j^+)$-switching shown in Figure~\ref{f:3}. 
The switching chooses vertices $(u_1,u_2,u_3)$ in $G'\in\Omega(\dvec')$ such that
\begin{enumerate}
\item[(a)] the five vertices $i,j,u_1,u_2,u_3$ are all distinct;
\item[(b)] $i\sim u_1$ and $u_2\sim u_3$;
\item[(c)] $u_1$ is not adjacent to $u_2$ and $u_3$ is not adjacent to $j$.
\end{enumerate}
The switching then replaces edges $iu_1$ and $u_2u_3$ by edges $u_1u_2$ and $u_3j$. The resulting graph $G$ belongs to $\Omega(\dvec)$.

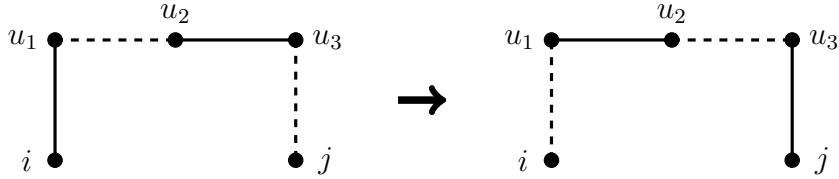
\begin{figure}[ht!]
\begin{center}
\begin{tikzpicture}[scale=0.8]
\draw [-,very thick,dashed] (0,2) -- (2,2) (4,2) -- (4,0);
\draw [-,very thick] (0,0) -- (0,2) (2,2) -- (4,2);
\draw [fill] (0,0) circle (0.12); \draw [fill] (0,2) circle (0.12);
\draw [fill] (2,2) circle (0.12); \draw [fill] (4,2) circle (0.12);
\draw [fill] (4,0) circle (0.12);
\node [left] at (-0.2,0) {$i$}; \node [right] at (4.2,0) {$j$};
\node [left] at (-0.1,2) {$u_1$}; \node [right] at (4.1,2) {$u_3$};
\node [above] at (2.0,2.1) {$u_2$};
\draw [->,line width = 1mm] (5.7,1) -- (6.5,1);
\begin{scope}[shift={(8.25,0)}]
\draw [-,very thick] (0,2) -- (2,2) (4,2) -- (4,0);
\draw [-,very thick,dashed] (0,0) -- (0,2) (2,2) -- (4,2);
\draw [fill] (0,0) circle (0.12); \draw [fill] (0,2) circle (0.12);
\draw [fill] (2,2) circle (0.12); \draw [fill] (4,2) circle (0.12);
\draw [fill] (4,0) circle (0.12);
\node [left] at (-0.2,0) {$i$}; \node [right] at (4.2,0) {$j$};
\node [left] at (-0.1,2) {$u_1$}; \node [right] at (4.1,2) {$u_3$};
\node [above] at (2.0,2.1) {$u_2$};
\end{scope}
\end{tikzpicture}
\caption{The $(i^-,j^+)$-switching used in Case 3.} 
\label{f:3}
\end{center}
\end{figure}

Trivially, we have $f_{(i^-,j^+)}(G')\le (d_i+1)M(\dvec)$, noticing that $d'_i=d_i+1$ and $M(\dvec')=M(\dvec)$. Given $u_1$, the number of choices of $(u_2,u_3)$ where a vertex collision occurs is at most $6\Delta(\dvec)$, noticing that $d'_j\le \Delta(\dvec)-1$. The number of choices where no vertex collision occurs but $u_1\sim u_2$ (or $u_3\sim j$) is at most $J(\dvec)-\Delta(\dvec)$. Hence, 
\[
f_{(i^-,j^+)}(G')\ge (d_i+1)(M(\dvec)-6\Delta(\dvec)-2(J(\dvec)-\Delta(\dvec)))=(d_i+1)(M(\dvec)-2J(\dvec)-4\Delta(\dvec)),
\]
which is positive by the theorem assumption.
Consequently, as $
b_{(i^-,j^+)}(G)\le d_jM(\dvec)$, 
we have
\[
\frac{|\Omega(\dvec')|}{|\Omega(\dvec)|} \le \frac{\max_{G\in \Omega(\dvec)} b_{(i^-,j^+)}(G)}{\min_{G'\in \Omega(\dvec')}f_{(i^-,j^+)}(G')} \le \frac{d_j M(\dvec)}{(d_i+1)(M(\dvec)-2J(\dvec)-4\Delta(\dvec))}<M(\dvec)^2.
\]
Combining all cases together completes the proof of the theorem.
\end{proof}

\noindent{\bf Remark.}\
Even though Case~2 is the reverse of Case~1, we did not use the reverse switching of Case~1 for Case 2. 
This is because the switching in Case~2 switches fewer edges than the reverse of the switching in Case~1, 
and as a result, it yields a better lower bound on $f_{(i^+,j^+)}(G')$. 
The reader may wonder why we did not use the reverse switching in Case~2 for Case~1. This is because in the 
reverse switching in Case~2, we have to restrict to $u_1\neq u_2$, and this restriction would result in a larger error and a useless lower bound on $f_{(i^-,j^-)}(G')$.

\section{Applications}
\label{sec:applications}

Theorem~\ref{thm:main} is tailored for treating heavy-tailed degree sequences,
and we discuss several such families below.  First we remark that as a straightforward
corollary of Theorem~\ref{thm:mixing2} we deduce that the switch chain
mixes in polynomial time for any $d$-regular degree sequence with
$d^*<n/2-9-28/d^*$, where $d^*=\min\{d,n-1-d\}$.   
(The parameter $d^*$ arises by complementation.)
The bound obtained on the mixing time is given by (\ref{mixingbound}).
However, rapid mixing was established for \emph{all} regular degree sequences by 
Cooper, Dyer and Greenhill~\cite{CDG}, and the upper bound on the mixing time given
in~\cite{CDG} is smaller than that obtained from (\ref{mixingbound}).

In~\cite{GW-power}, the first author and Wormald studied asymptotic enumeration of $\Omega(\dvec)$ for heavy-tailed degree sequences. A few particular families of heavy-tailed degree sequences were discussed in~\cite{GW-power}, including power-law sequences. In this section, we examine the mixing time of the switch chain on $\Omega(\dvec)$ for these families. 

A degree sequence $\dvec$ is said to be {\em nontrivial} if every component of $\dvec$ is positive. For sampling over $\Omega(\dvec)$, it is sufficient to consider only nontrivial $\dvec$.
Define 
\[ M_2(\dvec)=\sum_{i=1}^n d_i(d_i-1).\]
 We drop $\dvec$ from the notations $M(\dvec)$, $M_2(\dvec)$, $J(\dvec)$ and $\Delta(\dvec)$ when there is no confusion.  In all theorems of this section, we say a family of degree sequences 
is \emph{stable} if it is both 8-stable and strongly stable, and thus is also P-stable.

\subsection{Degree sequences with an upper bound on $M_2(\dvec)$}

The first example in~\cite{GW-power}  are degree sequences where $M_2$ does not grow too fast with $M$. An asymptotic enumeration result was given in~\cite{GW-power} when $M_2=o(M^{9/8})$. Under a much weaker condition we show a polynomial-time mixing bound for the switch chain on $\Omega(\dvec)$.

\begin{theorem}\label{M2-not-too-big}
Suppose that the graphical degree sequence $\dvec$ satisfies 
\begin{equation}
\label{condition1} 
M(\dvec)>2\sqrt{\Delta(\dvec)\big(M(\dvec)+M_2(\dvec)\big)}+18\Delta(\dvec)+56.
\end{equation}
Then 
the switch chain on $\Omega(\dvec)$
mixes in polynomial time, 
with mixing time bounded by~\emph{(\ref{mixingbound})}.  
Furthermore, the family of all degree sequences $\dvec$
which satisfy \emph{(\ref{condition1})} is stable.
\end{theorem}

\begin{proof} It is sufficient to verify {\tt Condition 1} of Theorem~\ref{thm:stable}.
Observe that
\[
M+M_2=\sum_{i=1}^n d_i^2 \ge \sum_{i=1}^{\Delta} d_i^2.
\]
By the Cauchy-Schwartz inequality, 
\[
\Delta\left(\sum_{i=1}^{\Delta} d_i^2\right) \ge \left(\sum_{i=1}^{\Delta} d_i\right)^2=J^2.
\]
Hence $J\le \sqrt{\Delta(M+M_2)}$. 
It follows that \Condone\ holds.
The assertions of the theorem follow from 
Theorems~\ref{thm:stable} and~\ref{thm:mixing2}.
\end{proof}

Since $\Delta \leq \sqrt{ 2 M_2}$, it follows from Theorem~\ref{M2-not-too-big} 
that the switch chain mixes in polynomial time whenever $M_2\le c M^{4/3}$ for some 
sufficiently small constant $c>0$.

\subsection{Power-law density-bounded sequences}\label{ss:power-density}

Two types of power-law degree sequences were considered in~\cite{GW-power}. We say $\dvec$ is a {\em power-law density-bounded}  sequence with parameter $\gamma$ if there exists $C>0$ such that the number of components in $\dvec$ having value $i$ is at most $Ci^{-\gamma}n$ for all $i\ge 1$ and for all $n$. This version of power-law degree sequences have been considered by many people in the literature, see for example~\cite{ChungLu}. 
An asymptotic enumeration result was given in~\cite{GW-power} for such $\dvec$ when $\gamma>5/2$. In the following theorem, we show that the switch chain mixes in polynomial time on $\Omega(\dvec)$ for such $\dvec$ if $\gamma>2$.

\begin{theorem}
Let $\dvec$ be a nontrivial power-law density-bounded degree sequence with parameter $\gamma>2$.
Then the switch chain on $\Omega(\dvec)$
mixes in polynomial time, 
with mixing time bounded by~\emph{(\ref{mixingbound})}.  
Furthermore, the family of all nontrivial power-law density-bounded degree sequences $\dvec$
with parameter $\gamma > 2$ is stable.
\label{thm:density-bounded}
\end{theorem}

\begin{proof} 
By definition, the maximum degree satisfies 
$\Delta\leq (Cn)^{1/\gamma}$. 
Let $j=cn^{1/\gamma}$ where $c>0$ is a sufficiently small constant.
Then $\sum_{i\ge j} n i^{-\gamma}\ge \Delta$. 
That is, there are at least $\Delta$ components of $\dvec$ which are bounded
below by $j$.
Hence
\[
J\le \sum_{i\ge j} C i^{1-\gamma} n = O(n^{2/\gamma}).
\]
Now $M=\Theta(n)$ as $\dvec$ is nontrivial.  Then $J+\Delta = o(M)$ as $\gamma > 2$. 
This implies that \Condone\ holds, and the statements of the theorem follow by 
Theorems~\ref{thm:stable} and~\ref{thm:mixing2}.
\end{proof}

When $\gamma>5/2$, a polynomial bound on the mixing time of the switch chain was given 
in~\cite{GS}.  To the best of our knowledge, the rest of Theorem~\ref{thm:density-bounded} is
new.

\subsection{Power-law distribution-bounded sequences}\label{ss:power-distribution}
\smallskip

Let $F_{\gamma}(i)=\sum_{j\ge i} j^{-\gamma} = O(i^{1-\gamma})$. We say that
$\dvec$ is a  {\em power-law distribution-bounded} sequence with parameter $\gamma$ if there exists $C>0$ such that the number of components of $\dvec$ taking value at least $i$ is at most $CF_{\gamma}(i)\, n$ for all $i$ and $n$. 
Power-law distribution-bounded sequences behave similarly to power-law density-bounded degree sequences, but allow longer tails (higher maximum degree). The maximum component of a power-law density-bounded sequence with parameter $\gamma$ is of order $n^{1/\gamma}$, whereas the maximum component of a power-law distribution-bounded sequence with parameter $\gamma$ is of order $n^{1/(\gamma-1)}$. Power-law distribution-bounded 
sequences, introduced in~\cite{GW-power}, are more realistic models for degree sequences of many real world networks.
In particular, this includes degree sequences composed of $n$ i.i.d.\ copies of power-law variables,
which were considered in~\cite{newman,vdH}. 

The following theorem bounds the mixing time of the switch chain on $\Omega(\dvec)$ where $\dvec$ is power-law distribution-bounded with parameter $\gamma>2$.

\begin{theorem}
Let $\dvec$ be a nontrivial power-law distribution-bounded degree sequence with parameter 
$\gamma>2$.
Then the switch chain on $\Omega(\dvec)$
mixes in polynomial time, 
with mixing time bounded by~\emph{(\ref{mixingbound})}.  
Furthermore, the family of all nontrivial power-law distribution-bounded degree sequences $\dvec$
with parameter $\gamma > 2$ stable.
\end{theorem}

\begin{proof} By definition, the maximum degree $\Delta$ is of order at most $n^{1/(\gamma-1)}$. 
It has been shown in~\cite[eq.~(54)]{Gao-power} that 
$J\leq (Cn)^{(2\gamma-3)/(\gamma-1)^2}$. 
Again we have $J + \Delta=o(M)$, which implies that \Condone\ holds when $\gamma>2$. 
The statements of the theorem follow by Theorems~\ref{thm:stable} and~\ref{thm:mixing2}.
\end{proof}

\subsection{Bi-regular degree sequences}

Another example examined in~\cite{GW-power} is $\dvec$ whose components only take two values $\delta$ and $\Delta$. This family is related to the ``multi-star graphs'' considered by Zhao~\cite{zhao},
if all star centres have the same degree.

\begin{theorem}
Assume that all components of $\dvec$ take value in $\{\delta,\Delta\}$ where $\delta\le \Delta$,
and let $\ell$ be the number of vertices with degree $\Delta$. 
Assume that one of the following conditions holds.
\begin{itemize}
\item[\emph{(a)}] $\ell\ge \Delta$ and $\Delta \ell +(n-\ell)\delta>2\Delta^2+18\Delta+56$.
\item[\emph{(b)}] $\ell<\Delta$ and $n\delta>(\Delta-\delta)\ell+2\delta\Delta+18\Delta+56$.
\end{itemize}
Then the switch chain on $\Omega(\dvec)$
mixes in polynomial time, 
with mixing time bounded by~\emph{(\ref{mixingbound})}.  
Furthermore, the family of all graphical degree sequences $\dvec$
which satisfy~\emph{(a)} or~\emph{(b)} is stable.
\end{theorem}

\begin{proof}
It is easy to check that if (a) or (b) holds then \Condone\ holds.
The assertions follow by Theorems~\ref{thm:stable} and~\ref{thm:mixing2}.
\end{proof}

\subsection{Long-tailed power-law degree sequences}

The first author and Wormald~\cite{GW-power} introduced {\em long-tailed power-law} degree sequences, which allow even longer tails than the power-law degree sequences. 
We say $\dvec$ follows a long-tailed power law with parameters $(\alpha,\beta,\gamma)$ if there is a constant $C>0$ such that for every $n$,
\begin{itemize}
\item each component of $\dvec$ is non-zero and either at most $C$, or at least $n^{\alpha}$;
\item for every integer $i\ge 1$, the number of components of $\dvec$ whose value is at least $in^{\alpha}$ but less than $(i+1)n^{\alpha}$ is at most $Cn^{\beta}i^{-\gamma}$.
\end{itemize}
Note that a long-tailed power-law degree sequence with $\alpha=0$ and $\beta=1$ is the same as a power-law density-bounded
degree sequence.

\begin{theorem}
Assume that a nontrivial graphical degree sequence $\dvec$ follows a long-tailed power law with 
parameters $(\alpha,\beta,\gamma)$.  
Further suppose that one of the following conditions holds:
\begin{itemize}
\item[\emph{(a)}] $1<\gamma\le 2$ and $\alpha+2\beta/\gamma<1$;
\item[\emph{(b)}] $\gamma>2$ and $\alpha+\beta<1$.
\end{itemize}
Then the switch chain on $\Omega(\dvec)$
mixes in polynomial time, 
with mixing time bounded by~\emph{(\ref{mixingbound})}.  
Furthermore, 
the family of all graphical degree sequences $\dvec$
which follow a long-tailed power law with parameters $(\alpha,\beta,\gamma)$
such that~(a) or~(b) holds is stable.
\end{theorem}

\begin{proof} As $\dvec$ is nontrivial, we have $M=\Omega(n)$. We will verify that if the parameters $(\alpha,\beta,\gamma)$ satisfy
(a) or (b) then $\Delta=o(n)$ and $J=o(n)$. 
This will imply that \Condone\ holds, and then the assertions of the theorem follow by Theorems~\ref{thm:stable} and~\ref{thm:mixing2}.

Let $\hat i=Cn^{\beta/\gamma}$ for some sufficiently large constant $C>0$. By definition, 
$\Delta=O(\hat i n^{\alpha})=O(n^{\alpha+\beta/\gamma})$. 
Let $\calH$ denote the set of vertices with degree at least $n^{\alpha}$ and let $H$ denote the sum 
of the degrees of the vertices in $\calH$. 
Suppose that
$ d_{\mu(1)}\geq d_{\mu(2)} \geq \cdots \geq d_{\mu(n)}$
for some permutation $\mu\in S_n$, and let
$\calL = \{ \mu(1),\ldots, \mu(\Delta)\}$
be the set of the first $\Delta$ vertices with respect to $\mu$.  Then  no vertex outside of
$\calL$ has higher degree than a vertex in $\calL$. 
Furthermore, $\calL$ contains at most all vertices in $\calH$ and $\Delta$ vertices in $[n]\setminus \calH$, whose degrees are at most $C$. It follows immediately that $J\le H+C\Delta$. 

First consider the case $1<\gamma\le 2$. The additional condition $\alpha+2\beta/\gamma<1$ ensures that $\Delta=o(n)$. We also have
\[
H=O(n^{\alpha+\beta}) \sum_{i=1}^{\hat i} i^{1-\gamma}=O( \log n) \cdot n^{\alpha+\beta} n^{(\beta/\gamma)(2-\gamma)}=O(\log n) \cdot n^{\alpha+2\beta/\gamma}=o(n),
\]
where the factor $\log n$ above accounts for the case $\gamma=2$. 
It follows then that $J\le H+C\Delta=o(n)$ when~(a) holds.

Next consider the case $\gamma>2$. The additional condition $\alpha+\beta<1$ ensures that $\Delta=o(n)$. Moreover, 
\[
H=O(n^{\alpha+\beta}) \sum_{i=1}^{\hat i} i^{1-\gamma} =O(n^{\alpha+\beta})=o(n).
\]
This implies that $J=o(n)$ when~(b) holds, completing the proof.
\end{proof}

\section{Comparison of different notions of stability}\label{sec:comparison}

Jerrum and Sinclair~\cite{JS90} initiated the study of stable degree sequences,
introducing the notion of P-stability.  
Given a graphical degree sequence $\dvec=(d_1,\ldots, d_n)$, let
\[ \Omega'(\dvec) = \bigcup_{\dvec'} \Omega(\dvec')\]
where $\dvec'$ ranges over the set of graphical degree sequences
$\dvec'=(d'_1,\ldots, d'_n)$ such that
$\norm{\dvec'-\dvec}_1\leq 2$ and $d_j'\leq d_j$ for all $j\in [n]$.

\begin{definition}[\cite{JS90}]\label{def:Pstable}
\emph{ 
A family $\mathcal{D}$ of degree sequences is \emph{P-stable} if there exists a polynomial $p$
such that for all positive integers $n$ and 
for every degree sequence $\dvec=(d_1,\ldots, d_n)\in\mathcal{D}$
\[ |\Omega'(\dvec)| \leq p(n)\, |\Omega(\dvec)|.\]
}
\end{definition}

Trivially, any finite set of degree sequences is P-stable. 
Recall that a degree sequence is nontrivial if every component is positive.
We now prove that P-stability is equivalent to 2-stability for families of nontrivial degree sequences. 

\begin{proposition}\label{p:equivalence} 
A family of nontrivial degree sequence is P-stable if and only if it is 2-stable.
 \end{proposition}

\begin{proof} 
Observe that $\Omega'(\dvec)\subseteq
{\cal N}_2(\dvec)$, and that $\Omega'(\dvec)$ is
a union over at most $n^2$ degree sequences $\dvec'$.
Therefore, if ${\cal D}$ is 2-stable then
there exists some constant $\alpha > 0$ such that 
\[ |\Omega'(\dvec)| \leq n^2 M(\dvec)^\alpha\, |\Omega(\dvec)|
  \leq n^{2+2\lceil \alpha\rceil}\, |\Omega(\dvec)|.\]
This shows that ${\cal D}$ is P-stable with polynomial
$p(n) = n^{2+2\lceil \alpha\rceil}$,  and hence 2-stability
implies P-stability.

For the converse, let ${\cal D}$ be a P-stable family of nontrivial degree sequences
with polynomial $p(n)$. 
For any $\dvec\in {\cal D}$ we have $M(\dvec)\geq n$, since $\dvec$ is nontrivial.
We will show that 
\begin{equation}
\label{sufficient}
|\Omega(\dvec')|\le n^2  |\Omega'(\dvec)| \quad \text{ for all
$\dvec\in {\cal D}$ and $\dvec'\in {\cal N}_2(\dvec)$.} 
\end{equation}
This yields our proposition as 
then, for all $\dvec\in\mathcal{D}$ and all $\dvec'\in N_2(\dvec)$,
\[
|\Omega(\dvec')| \leq n^2\, |\Omega'(\dvec)|
 \leq n^2\, p(n)\, |\Omega(\dvec)| \le M(\dvec)^{\alpha} |\Omega(\dvec)|
\]
for some constant $\alpha>0$, using P-stability for the second inequality. 
It is sufficient to consider $\dvec'$ such that $d'_i\ge d_i+1$ for some $i$, as otherwise $\Omega(\dvec')\subseteq \Omega'(\dvec)$. 

\medskip

{\em Case 1: $d'_i=d_i+1$ and $d'_j=d_j+1$ for some $i\neq j$.} Define a switching from $\Omega(\dvec')$ to $\Omega'(\dvec)$ as follows. 
For any $G\in \Omega(\dvec')$, the switching deletes an edge incident to $i$ and an edge incident to $j$. The resulting graph $G'$ is in  $\Omega'(\dvec)$. 
Each $G$ can be switched to at least one graph in $\Omega'(\dvec)$, and each 
graph in $\Omega'(\dvec)$ can be produced by at most $n^2$ graphs in 
$\Omega(\dvec')$ using this switching. 
Hence $|\Omega(\dvec')|\le n^2  |\Omega'(\dvec)|$.
Therefore (\ref{sufficient}) holds in this case.

\smallskip

The other cases where $d_i'=d_i+2$ for some $i$, or $d'_i=d_i+1$ and $d'_j=d_j-1$ for some $i\neq j$ can be analysed in a similar way. If $d_i'=d_i+2$
for some $i$ then the switching chooses and deletes two distinct
edges incident with $i$, while if
$d'_i=d_i+1$ and $d'_j=d_j-1$ then the
switching deletes an edge incident with $i$. 
In both cases, each graph in $\Omega'(\dvec)$ can be
produced by at most $n$ graphs in $\Omega(\dvec')$ in this way.
Therefore (\ref{sufficient}) also holds in these cases, completing the proof.
\end{proof}

\smallskip

The following proposition follows immediately by~(\ref{k}) and the equivalence of 2-stability and P-stability.

\begin{proposition}\label{p:stable}
 If a family of degree sequences is 8-stable then it is P-stable.
 \end{proposition}

A slightly adjusted definition of P-stability was studied by Jerrum, Sinclair and McKay~\cite{JMS}.  We 
remark that the definition of P-stability from~\cite{JMS} is equivalent to 2-stability, and thus is equivalent 
to the original
Jerrum--Sinclair definition of P-stability. The equivalence can be proved by a similar argument as in 
Proposition~\ref{p:equivalence}.

\bigskip

Jerrum and Sinclair~\cite{JS90} introduced the idea of P-stability
in order to prove a rapid mixing result for a Markov chain for
sampling graphs with given degree sequence.
This Markov chain, which we will call the \emph{JS chain},
has state space $\Omega'(\dvec)$, where $\dvec$ is the target degree
sequence.  The JS chain produces an almost uniform sample from $\Omega'(\dvec)$,
which is rejected if the output is not an element of
$\Omega(\dvec)$.  The expected number of restarts required
is polynomial if and only if 
the ratio $|\Omega'(\dvec)|/|\Omega(\dvec)|$ is bounded above by some polynomial in $n$,
leading to the definition of P-stable families.

For our purposes, we just need to know 
that the possible transitions $G\mapsto G'$ of the JS chain are as follows:
\begin{itemize}
\item (\emph{insertion}):  $G'$ is obtained from $G$ by inserting
an edge.  Here $G'$ has degree $\dvec$.
\item (\emph{deletion}):  $G'$ is obtained from $G$ by deleting an
edge of $G$.  Here $G$ has degree $\dvec$.
\item (\emph{hinge-flip}): $G'$ is obtained from $G$ by deleting
an edge $uv$ and inserting an edge $wv$.  Here neither $G$ nor $G'$ have
degree $\dvec$, as vertex $w$ has degree $d_w-1$ in $G$ and
vertex $u$ has degree $d_u-1$ in $G'$.
\end{itemize}
All such pairs $(G,G')$ correspond to a transition of the JS chain, 
and there are no other transitions.
(The terminology ``hinge-flip'' comes from~\cite{AK}.)

Amanatidis and Kleer~\cite{AK} defined a stronger notion of
stability, which they called \emph{strong stability}, as follows.
Say that two graphs $G, H$ are at distance $r$ in the JS chain if
$H$ can be obtained from $G$ using at most $r$ transitions.
Let $\operatorname{dist}(G,\dvec)$ be the minimum distance of
$G$ from an element of $\Omega(\dvec)$, and define
\[ k_{\operatorname{JS}}(\dvec) = \max_{G'\in\Omega'(\dvec)}\, 
   \operatorname{dist}(G,\dvec).
\]
That is, from any element of $\Omega'(\dvec)$ it is possible to
reach an element of $\Omega(\dvec)$ using at most $r$ transitions of
the JS chain.

\begin{definition}
A family $\mathcal{D}$ of graphical degree sequences is strongly stable
if there is a constant $\ell$ such that $k_{\operatorname{JS}}(\dvec)\leq \ell$
for all $\dvec\in\mathcal{D}$.
\end{definition}

Amanatidis and Kleer proved that every strongly stable family is 
P-stable~\cite[Proposition~3]{AK}, and that the switch chain has
polynomial mixing time for all degree sequences from a strongly stable 
family~\cite[Theorem~4]{AK}.  Hence, the strong stability is a (possibly) 
stronger notion than the P-stability. 
However, we do not know the relationship between P-stability and $k$-stability for $k>2$. 
To conclude, we have the following relations between the various notions of stability under discussions:

\smallskip

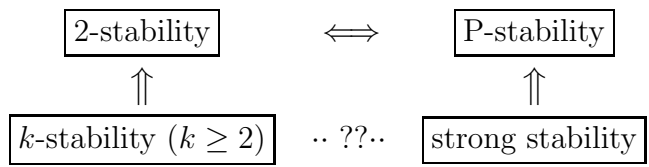
\begin{center}
\renewcommand{\arraystretch}{1.4}
\begin{tabular}{ccc}
\framebox{2-stability} & \, $\Longleftrightarrow$ \quad & \framebox{P-stability}\\
$\big\Uparrow$ & & $\big\Uparrow$\\
\framebox{$k$-stability ($k\geq 2$)} & $\cdot\cdot$ ??$\cdot\cdot$  & \framebox{strong stability}
\end{tabular}
\end{center}

\medskip

Next, we prove that \Condtwo, which is a sufficient condition for $(2,2)$-stability, also implies strongly stability. 

\begin{proof}[Proof of Theorem~\ref{thm:stable}(b)]
It suffices to observe that each switching operation used in the proof of 
Theorem~\ref{thm:main} can be implemented using a sequence of at most three
transitions of the Jerrum-Sinclair chain.  This implies that $\mathcal{D}$
is strongly stable, as every $\dvec\in\mathcal{D}$ satisfies
$k_{\operatorname{JS}}(\dvec)\leq 3$.
Specifically, 
an $(i^-,j^-)$-switching (Case~1) can be implemented
by a deletion followed by two hinge flips, 
an $(i^+,j^+)$-switching (Case~2) can be implemented by a hinge flip
followed by an insertion,  and an $(i^-,j^+)$-switching
(Case~3) can be implemented using two hinge flips.
\end{proof}

\section{Directed graphs}
\label{sec:directed}

A \emph{directed graph} (or \emph{digraph}) $G=(V,A)$ consists
of a set $V$ of vertices and a set $A = A(G)$ of \emph{arcs}
(directed edges).
A \emph{directed degree sequence} is a pair $(\dvec^-,\, \dvec^+)$ of 
sequences of nonnegative integers,
\[ \dvec^- = (d_1^-,\ldots, d_n^-),\qquad \dvec^+ = (d_1^+,\ldots, d_n^+)\]
such that  $\sum_{j=1}^n d_j^- = \sum_{j=1}^n d_j^+$.  The sequence
is \emph{digraphical} if there exists a directed graph with vertex set $[n]$
such that
$d_j^-$ is the in-degree and $d_j^+$ is the out-degree of vertex $j$,
for all $j\in [n]$. 
Write $\Omegad$ for the set of all directed graphs with directed degree
sequence $(\dvec^-,\dvec^+)$.  Note that loops are not permitted,
so a directed graph (digraph) in $\Omegad$ corresponds to a bipartite
graph  on vertex set $\{u_1,\ldots, u_n\}\cup \{v_1,\ldots, v_n\}$
such that $u_j$ has degree $d_j^-$ and $v_j$ has degree $d_j^+$
and $(u_j,v_j)$ is not an edge, for $j\in [n]$.
This connection between bipartite graphs and directed graphs is very
well-known and has a long history~\cite{G57,pet}.
Again, we are interested in sequences of directed degree sequences
$(\dvec^-(n),\dvec^+(n))$ indexed by $n$, but usually drop the dependence
on $n$ from our notation.
We note that Erd{\H o}s et al.\ gave a definition
of P-stability for directed degree sequences, 
see~\cite[Definition~7.4]{hungarians}.   We will use a variant of this definition.

The directed switch Markov chain, denoted by $\Md$, has state space
$\Omegad$ and transitions described by the following procedure:
from the current digraph $G\in\Omegad$, choose an unordered pair
$\{ (i,j),(k,\ell)\}$ of distinct arcs of $G$ uniformly at random. 
If $i,j,k,\ell$ are distinct and $\{ (i,j),(k,\ell)\} \cap A(G) = \emptyset$ 
then delete the arcs $(i,j)$, $(k,\ell)$ from $G$ and add the arcs $(i,\ell)$, $(k,j)$ 
to obtain the new state; otherwise, remain at $G$.

The directed switch chain is symmetric but it is not always irreducible, 
unlike the switch chain for directed graphs.
We say that the directed degree sequence $\vecdvec$ is 
\emph{switch-irreducible} if the directed switch chain on $\Omegad$
is irreducible.  The directed switch chain is ergodic on $\Omegad$
for any switch-irreducible digraphical degree sequence $\vecdvec$,
with uniform stationary distribution.
Berger and M{\" u}ller-Hanneman~\cite{BMH} and LaMar~\cite{lamar} provide
characterisations which can be applied
to test whether a given directed degree sequence $\vecdvec$
is switch-irreducible. 
Rather than restricting to switch-irreducible sequences, some authors
allow the directed switch chain to occasionally perform an additional operation, known as a triple swap~\cite{EMMS}, which reverses the arcs of
a directed 3-cycle.  With the addition of this operation, the chain is irreducible for all directed degree sequences.  However, we use the directed
switch chain as defined above, with no triple swaps.

Let
\[ M(\dvec^-,\dvec^+) = \sum_{j=1}^n d_j^- = \sum_{j=1}^n d_j^+\]
be the number of arcs in a directed graph with directed degree sequence $\vecdvec$,
and let 
\[
\Delta^-(\dvec^-) = \max\{d_1^-,\ldots d_n^-\}, \quad 
  \Delta^+(\dvec^+) = \max\{ d_1^+,\ldots d_n^+\}. 
\]
Write
\[
\Delta(\dvec^-,\dvec^+)=\max\{\Delta^-(\dvec^-),\, \Delta^+(\dvec^+)\}.
\]
For any positive integer $k$ and nonnegative real number $\alpha$,
we say that 
the digraphical degree sequence $\vecdvec$ is $(k,\alpha)$\emph{-stable} if
\[ |\Omega(\dvec^{'-},\dvec^{'+})|\leq M(\dvec^-,\dvec^+)^\alpha\, |\Omegad|\]
for every digraphical degree sequence $(\dvec^{'-},\dvec^{'+})$ such that
\[ \norm{\dvec^{'-}-\dvec^-}_1 + \norm{\dvec^{'+}-\dvec^+}_1 \leq k.\]
We say a family of digraphical degree sequences ${\cal D}$ is $k$-\emph{stable} if there exists a fixed $\alpha>0$ such that all degree sequences in ${\cal D}$ are $(k,\alpha)$-stable.

We adapt the definition of P-stability for directed degree sequences
from~\cite[Definition~7.4]{hungarians}.  
Given a digraphical degree sequence $\vecdvec$, let
\[ \Omega'(\dvec^-,\dvec^+) = \bigcup_{(\dvec^{'-},\dvec^{'+})}
    \Omega(\dvec^{'-},\dvec^{'+})\]
where $(\dvec^{'-},\dvec^{'+})$ ranges over the set of digraphical degree sequences
with $d_j^{'-}\leq d_j^-$ and $d_j^{'+}\leq d_j^+$ for all $j\in [n]$,
such that 
\[ \norm{\dvec^{'-} - \dvec^-}_1+\norm{\dvec^{'+} - \dvec^+}_1\leq 2.
\]

\begin{definition}\label{def:Pstable-directed}
\emph{ 
A family $\mathcal{D}$ of directed degree sequences is \emph{P-stable} if there exists a 
polynomial $p$ such that for all positive integers $n$ and 
for every directed degree sequence $\vecdvec\in\mathcal{D}$, 
\[ |\Omega'(\dvec^-,\dvec^+)| \leq p(n)\, |\Omegad|.\]
}
\end{definition}

Note that a family of digraphical degree sequences is P-stable 
if and only if it is 2-stable, as can be shown by
adapting the proof of Proposition~\ref{p:equivalence}.
We will establish the directed analogue of Theorem~\ref{thm:mixing}, stated below.

\begin{theorem}\label{thm:mixing-directed}
Let $\vecdvec$ be a digraphical switch-irreducible directed degree sequence. 
Write $M=M(\dvec^-,\dvec^+)$ and $\Delta=\Delta(\dvec^-,\dvec^+)$.
If $\vecdvec$ is $(12,\alpha)$-stable
then the directed switch chain on $\Omegad$ mixes in polynomial time, with
mixing time $\tau(\varepsilon)$ which satisfies
\[
\tau(\varepsilon) \leq  1200\, \Delta^{16}\,n^{6}\, M^{3+\alpha} \, 
   \Big( M\log M + \log(\varepsilon^{-1})\Big).
\]
\end{theorem}

Given a directed degree sequence $\vecdvec$, suppose that $\rho,\xi\in S_n$ are 
permutations such that
$d^-_{\mu(1)}\geq d^-_{\mu(2)}\geq \cdots \geq d^-_{\mu(n)}$ and
$d^+_{\xi(1)}\geq d^+_{\xi(2)}\geq \cdots \geq d^+_{\xi(n)}$. 
Let
\[ J^-(\dvec^-,\dvec^+) = \sum_{\ell=1}^{\Delta^-(\dvec^-)} d^+_{\xi(\ell)},\qquad
 J^+(\dvec^-,\dvec^+) = \sum_{\ell=1}^{\Delta^+(\dvec^+)} d^-_{\mu(\ell)}.\]
These will turn out to be the appropriate directed analogues of the parameter
$J(\dvec$) used for undirected graphs. When it causes no confusion, we drop $(\dvec^-,\dvec^+)$ from 
notation such at $M(\dvec^-,\dvec^+)$, etc.

\begin{theorem}\label{thm:stable-directed}
Let $\dvec$ be a digraphical directed degree sequence which satisfies
\begin{equation}
\label{big-condition-directed}
 M> J^- + J^++ 8\big(\Delta^-+ \Delta^+\big) + 48.  
\end{equation} 
Then $\dvec$ is $(12,12)$-stable.  
\end{theorem}

Again, we have not attempted to optimise the coefficients in (\ref{big-condition-directed}).
Combining Theorems~\ref{thm:mixing-directed} and~\ref{thm:stable-directed}
immediately establishes the following.

\begin{theorem}\label{thm:mixing2-directed} Assume that $\vecdvec$ is
a digraphical switch-irreducible directed degree sequence which satisfies 
\emph{(\ref{big-condition-directed})}.
Then the switch chain on $\Omegad$ mixes in polynomial time, with 
mixing time $\tau(\varepsilon)$ which satisfies
\begin{equation}
\tau(\varepsilon) \leq 600\, \Delta^{16}\,n^{6}\, M^{15}\,  
   \Big( M\log M + \log(\varepsilon^{-1})\Big). \label{mixingbound-directed}
\end{equation}
 
\end{theorem}

We do not believe that the mixing time bounds in Theorem~\ref{thm:mixing-directed} and 
Theorem~\ref{thm:mixing2-directed} are tight.  

\subsection{Multicommodity flow and mixing time}

In~\cite{directed,GS}, the mixing time of the directed switch chain
is analysed using a multicommodity flow argument.  The flow is an extension
of the one used in~\cite{CDG}, following the steps described in Section~\ref{s:switch-chain}
in the undirected case.    Again, a set $\Psi(G,G')$ of pairings
is defined for each $(G,G')\in\Omegad^2$, and a canonical path
$\gamma_\psi(G,G')$ is defined from $G$ to $G'$, where each step of the path
is a transition of the directed switch chain.
For our purposes, it is sufficient to understand
the structure of the encodings which arise along these paths.

\begin{definition}\label{def:encoding-directed}
An encoding $L$ of a digraph $Z\in\Omegad$ is an arc-labelled digraph on $n$ vertices,
with arc labels in $\{-1,1,2\}$, such that
\begin{itemize}
\item[(i)] the sum of arc-labels on arcs into vertex $j$ equals $d_j^-$
and the sum of arc-labels on arcs out of vertex $j$ equals $d_j^+$,
for all $j\in [n]$;
\item[(ii)] the arcs with labels $-1$ or $2$ form a subdigraph of one of the
digraphs shown in Figure~\ref{f:directed-zoo}.
\end{itemize}
An arc labelled $-1$ or $2$ is a \emph{defect arc}.  
\end{definition}

Given $G,G',Z\in\Omegad$, identify each of $Z, G, G'$ with their 0-1 adjacency
matrix and define the matrix $L$ by 
\[ L + Z = G + G',\]
just as we did for undirected graphs. Then $L$ corresponds to an arc-labelled
digraph, which we also denote by $L$.  It follows from the next result that the
arc-labelled digraph $L$ is an encoding, which we call the encoding of $Z$
with respect to $(G,G')$.

\begin{lemma}[{\cite[Lemma 3.3(ii)]{GS}}] \label{directed-configuration}
Given $G,G' \in \Omegad$ with symmetric difference $G \triangle G'$, 
let $(Z,Z')$ be a transition on the canonical path from $G$ to $G'$ with 
respect to the pairing $\psi \in \Psi(G,G')$. 
Let $L$ be the encoding of $Z$ with respect to $(G,G')$. Then 
there are at most five defect arcs in $L$.
The digraph consisting of the defect arcs in $L$ must form a subdigraph of 
one of the possible labelled digraphs shown in Figure~\ref{f:directed-zoo},
up to the symmetries described below.
\end{lemma}

Define the \emph{arc-reversal} operator $\zeta$,
which acts on a digraph
$G$ by reversing every arc in $G$; that is, replacing $(u,v)$ by $(v,u)$
for every arc $(u,v)\in A(G)$.  
In Figure~\ref{f:directed-zoo}, $\{ \mu,\nu\} = \{ -1,2\}$ and $\{ \xi,\omega\} = \{-1,2\}$ independently, giving four symmetries obtained by exchanging these pairs. 
We can also apply the operation $\zeta$ to reverse the orientation of all arcs. 
Hence each digraph shown in Figure~\ref{f:directed-zoo} represents up to eight 
possible digraphs. 

\begin{figure}[ht!] 
\begin{center}
\begin{tikzpicture}[scale = 1.1]
\draw[fill] (1,2) circle (0.1);
\draw[fill] (1,1) circle (0.1);
\draw[fill] (1,3) circle (0.1);
\draw[fill] (0,2) circle (0.1);
\draw[fill] (2,2) circle (0.1);
\draw[fill] (2,1) circle (0.1);
\draw[thick,->] (0.85,2) -- (0.15,2);
\draw[thick,->] (1,2.15) -- (1,2.85);
\draw[thick,->] (1,1.85) -- (1,1.15);
\draw[thick,->] (1.85,2) -- (1.15,2);
\draw[thick,->] (2,1.85) -- (2,1.15);
\node [right] at (1,2.5) {$\mu$};
\node [above] at (0.5,2) {$\mu$};
\node [left] at (1,1.5) {$\nu$};
\node [above] at (1.5,2) {$\omega$};
\node [right] at (2,1.5) {$\xi$};
\draw[fill] (3,2) circle (0.1);
\draw[fill] (4,2) circle (0.1);
\draw[fill] (5,2) circle (0.1);
\draw[fill] (4,1) circle (0.1);
\draw[fill] (4,3) circle (0.1);
\draw[thick,->] (3.85,2) -- (3.15,2);
\draw[thick,->] (4,2.15) -- (4,2.85);
\draw[thick,->] (4,1.85) -- (4,1.15);
\draw[thick,->] (4.85,2) -- (4.15,2);
\draw[thick,->] (4.9,1.9) -- (4.1,1.1);
\node [right] at (4,2.5) {$\mu$};
\node [above] at (3.5,2) {$\mu$};
\node [left] at (4,1.5) {$\nu$};
\node [above] at (4.5,2) {$\omega$};
\node [right] at (4.5,1.5) {$\xi$};
\draw[fill] (6,2) circle (0.1);
\draw[fill] (7,2) circle (0.1);
\draw[fill] (8,2) circle (0.1);
\draw[fill] (7,1) circle (0.1);
\draw[fill] (7,3) circle (0.1);
\draw[thick,->] (6.85,2) -- (6.15,2);
\draw[thick,->] (7,2.15) -- (7,2.85);
\draw[thick,->] (7,1.85) -- (7,1.15);
\draw[thick,->] (7.85,2) -- (7.15,2);
\draw[thick,->] (7.9,1.9) -- (7.1,1.1);
\node [right] at (7,2.5) {$\mu$};
\node [above] at (6.5,2) {$\nu$};
\node [left] at (7,1.5) {$\mu$};
\node [above] at (7.5,2) {$\omega$};
\node [right] at (7.5,1.5) {$\xi$};
\draw[fill] (9,2) circle (0.1);
\draw[fill] (10,2) circle (0.1);
\draw[fill] (9,1) circle (0.1);
\draw[fill] (10,1) circle (0.1);
\draw[fill] (10,3) circle (0.1);
\draw[thick,->] (9.85,2) -- (9.15,2);
\draw[thick,->] (10,2.15) -- (10,2.85);
\draw[thick,->] (9.95,1.85) -- (9.95,1.15);
\draw[thick,->] (10.05,1.15) -- (10.05,1.85);
\draw[thick,->] (9.85,1) -- (9.15,1);
\node [right] at (10,2.5) {$\mu$};
\node [above] at (9.5,2) {$\mu$};
\node [left] at (10,1.5) {$\nu$};
\node [right] at (10,1.5) {$\omega$};
\node [below] at (9.5,1) {$\xi$};
\begin{scope}[shift={(0,0.5)}]
\draw[fill] (1.5,-1) circle (0.1);
\draw[fill] (1.5,-2) circle (0.1);
\draw[fill] (1.5,-3) circle (0.1);
\draw[fill] (0.5,-2) circle (0.1);
\draw[fill] (0.5,-3) circle (0.1);
\draw[thick,->] (1.35,-2) -- (0.65,-2);
\draw[thick,->] (1.5,-1.85) -- (1.5,-1.15);
\draw[thick,->] (1.45,-2.15) -- (1.45,-2.85);
\draw[thick,->] (1.55,-2.85) -- (1.55,-2.15);
\draw[thick,->] (1.35,-3) -- (0.65,-3);
\node [right] at (1.5,-1.5) {$\mu$};
\node [above] at (1,-2) {$\nu$};
\node [left] at (1.5,-2.5) {$\mu$};
\node [right] at (1.5,-2.5) {$\omega$};
\node [below] at (1,-3) {$\xi$};
\draw[fill] (3.5,-2) circle (0.1);
\draw[fill] (4.5,-1) circle (0.1);
\draw[fill] (4.5,-2) circle (0.1);
\draw[fill] (4.5,-3) circle (0.1);
\draw[thick,->] (4.35,-2) -- (3.65,-2);
\draw[thick,->] (4.5,-1.85) -- (4.5,-1.15);
\draw[thick,->] (4.45,-2.15) -- (4.45,-2.85);
\draw[thick,->] (4.55,-2.85) -- (4.55,-2.15);
\draw[thick,->] (4.4,-2.9) -- (3.6,-2.1);
\node [right] at (4.5,-1.5) {$\mu$};
\node [above] at (4,-2) {$\mu$};
\node [left] at (4.5,-2.4) {$\nu$};
\node [right] at (4.5,-2.5) {$\omega$};
\node [left] at (4,-2.6) {$\xi$};
\draw[fill] (6.5,-2) circle (0.1);
\draw[fill] (7.5,-1) circle (0.1);
\draw[fill] (7.5,-2) circle (0.1);
\draw[fill] (7.5,-3) circle (0.1);
\draw[thick,->] (7.35,-2) -- (6.65,-2);
\draw[thick,->] (7.5,-1.85) -- (7.5,-1.15);
\draw[thick,->] (7.45,-2.15) -- (7.45,-2.85);
\draw[thick,->] (7.55,-2.85) -- (7.55,-2.15);
\draw[thick,->] (7.4,-2.9) -- (6.6,-2.1);
\node [right] at (7.5,-1.5) {$\mu$};
\node [above] at (7,-2) {$\nu$};
\node [left] at (7.5,-2.4) {$\mu$};
\node [right] at (7.5,-2.5) {$\omega$};
\node [left] at (7,-2.6) {$\xi$};
\draw[fill] (9,-2) circle (0.1);
\draw[fill] (10,-1) circle (0.1);
\draw[fill] (10,-2) circle (0.1);
\draw[fill] (10,-3) circle (0.1);
\draw[thick,->] (9.85,-2) -- (9.15,-2);
\draw[thick,->] (10,-1.85) -- (10,-1.15);
\draw[thick,->] (9.95,-2.15) -- (9.95,-2.85);
\draw[thick,->] (10.05,-2.85) -- (10.05,-2.15);
\draw[thick,->] (9.9,-2.9) -- (9.1,-2.1);
\node [right] at (10,-1.5) {$\nu$};
\node [above] at (9.5,-2) {$\mu$};
\node [left] at (10,-2.4) {$\mu$};
\node [right] at (10,-2.5) {$\omega$};
\node [left] at (9.5,-2.6) {$\xi$};
\end{scope}
%%%
\draw[-] (2.5,-3.1) -- (2.5,3.1);
\draw[-] (5.5,-3.1) -- (5.5,3.1);
\draw[-] (8.5,-3.1) -- (8.5,3.1);
\draw[-] (-0.1,0) -- (10.6,0);
%%%
\end{tikzpicture}
\caption{Possible configurations of defect arcs, up to symmetries.
See~\cite[Lemma~3.3(ii)]{GS}}.
\label{f:directed-zoo}
\end{center}
\end{figure}
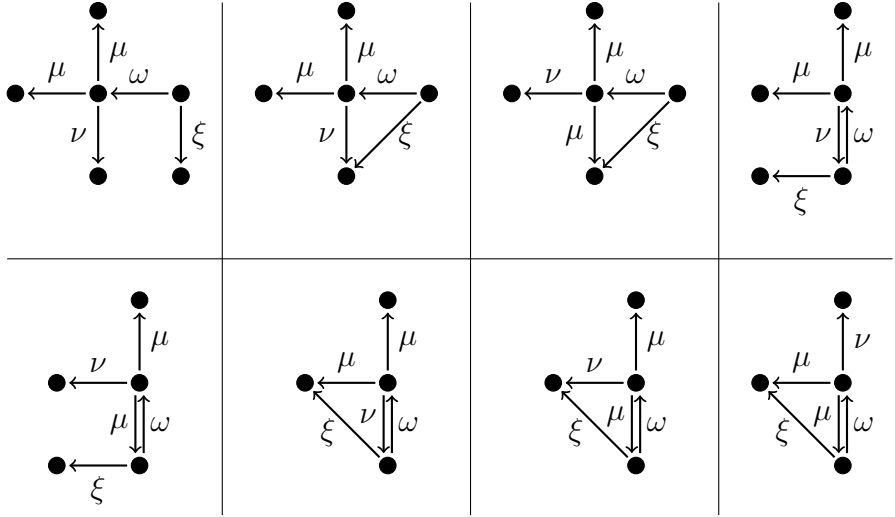

\medskip

Given $Z\in\Omegad$,
we say that encoding $L$ is \emph{consistent with} $Z$ if $L+Z$ only takes
entries in $\{0,1,2\}$.  Let $\mathcal{L}(Z)$ be the set of encodings which
are consistent with $Z$ and satisfy (i), (ii) from Definition~\ref{def:encoding-directed} above.
This is the same set of encodings used in~\cite{GS}.

\begin{theorem}\label{thm:catherine-directed}
Let $\vecdvec$ be a digraphical switch-irreducible directed degree sequence. 
Suppose that there exists 
a function $g(\dvec^-,\dvec^+)$, which depends only on $\vecdvec$, such that
\[
 |\mathcal{L}(Z)| \leq   g(\dvec^-,\dvec^+)\, |\Omegad|
\]
for all $Z\in\Omegad$.
Then the mixing time $\tau(\varepsilon)$ of the switch chain on $\Omegad$ 
satisfies
   \[
   \tau(\varepsilon) \leq 2 g(\dvec^-,\dvec^+)\,\Delta^{16}\, M^3\,
   \Big(  M\log M + \log(\varepsilon^{-1})\Big).
   \]
\label{Lbound-directed}
\end{theorem}

\begin{proof}
We follow the structure of the proof of~\cite[Theorem 1.2]{GS}.
It follows from the bipartite model of directed graphs that
\[ |\Omegad| \leq M! \leq \sqrt{2\pi M}\, \left(\frac{M}{e}\right)^M.\]
Therefore the smallest stationary probability $\pi^*$ satisfies
 $\log( 1/\pi*) = \log|\Omegad| \leq M\log M$.
Next, observe that $\ell(f)\leq M$ since each transition along a canonical
path from $G$ to $G'$ replaces an edge of $G$ by an edge of $G'$. 
Finally, if $e=(Z,Z')$ is a transition of the directed switch chain then
$1/Q(e) = \binom{M}{2}\, |\Omegad|\leq \nfrac{1}{2} M^2\, |\Omegad|$.  
Given a transition $e=(Z,Z')$ of the directed switch chain, 
the load $f(e)$ on the transition satisfies  
\[ f(e) \leq 4 \, \Delta^{16} \,\frac{|\mathcal{L}(Z)|}{|\Omegad|^2},\]
as proved in~\cite[Lemma 3.4]{GS}. 
(Again, the degree condition in the statement of~\cite[Theorem 1.2]{GS} is
only needed in the proof of the critical lemma,
so this bound on $f(e)$ holds even when $\vecdvec$ does not satisfy that
condition.
This gives
\[ \rho(f) = \max_e \frac{f(e)}{Q(e)} \leq 
2 \Delta^{16} M^{2} \, 
   \max_{Z\in\Omegad} \frac{|\mathcal{L}(Z)|}{|\Omegad|}\leq
2 g(\dvec^-,\dvec^+)\, \Delta^{16} M^{2}.
\]
Substituting these expressions into (\ref{flowbound}) completes the
proof.
\end{proof}

\begin{lemma}\label{lem:counting-directed}
Assume that the graphical degree sequence $\dvec$ is $(12,\alpha)$-stable
for some nonnegative real number $\alpha$. Then for any $Z\in\Omegad$,
\[  |\mathcal{L}(Z)|\le  600\, n^6 M^{\alpha}\, |\Omegad|.\] 
\end{lemma}

\begin{proof}
By definition, $|\mathcal{L}(Z)|$ is the number of arc-labelled
directed graphs which satisfy conditions (i) and (ii), as well as some
other constraints.   First we bound the number of ways to choose
the defect arcs of $Z$, by first choosing a subdigraph $H$ of one
of the digraphs in Figure~\ref{f:directed-zoo}, and then mapping each
vertex of $H$ to a distinct vertex of $Z$.
There are 64 different digraphs shown
in Figure~\ref{f:directed-zoo}, after considering all symmetries.
Each of these digraphs has 5 arcs and at most 6 vertices.
If a subdigraph of $H$ has $\ell$ vertices then the number of
injective functions $\varphi:V(H) \to [n]$ is at most $n^\ell$.
By considering the maximum number
of vertices in a subdigraph consisting of a given number of arcs,
we find that the number of ways to choose $H$ and $\varphi$
is at most
\begin{align*}
  64\Big( n^6 + 5n^6 + 10 n^5 + 10 n^4 + 5n^2 + 1\Big)
 & \leq 64n^6 \Big(1 + 5 + \dfrac{5}{2} + \dfrac{5}{8} + \dfrac{5}{256} + \dfrac{1}{4096}
   \Big)\\
 &\leq 600\, n^6.
\end{align*}
The first inequality follows since $n\geq 4$, as otherwise there are no switch operations available.

Let $\mathcal{E}$ denote the chosen defect arcs with their labels.
We now bound the number of ways to complete $\mathcal{E}$ to an encoding $L$.
All arcs in $L\setminus\mathcal{E}$ are labelled 1,
and the number of arcs into vertex $j$ in $L\setminus \mathcal{E}$ equals
$d^{'-}_j=d^-_j - x^-_j$, where $x^-_j$ is the sum of arc-labels from arcs
into $j$ which belong to $\mathcal{E}$.
Similarly, the number of arcs out of vertex $j$ in $L\setminus\mathcal{E}$
equals $d^{'+}_j = d^+_j-x^+_j$, for all $j\in [n]$.
Note that $\norm{ \dvec^{'-}}_1 = \norm{\dvec^{'+}}_1$. 
Hence the number of valid encodings $L$ given $\mathcal{E}$ is at most
$|\Omega(\dvec^{'-},\dvec^{'+})|$.  
By considering cases we see that for all possible sets $\mathcal{E}$,
\[ \norm{\dvec^{'-} - \dvec^-}_1 + \norm{\dvec^{'+} - \dvec^+}_1 \leq 12.
\]
Since $\vecdvec$ is $(12,\alpha)$-stable, the result follows.
\end{proof}

\subsection{Sufficient condition for stability}

Assume that $\vecdvec$ is a digraphical degree sequence. For positive integers $k$, let 
\[ \N_k(\dvec^-,\dvec^+)=
  \{ (\dvec^{'-},\dvec^{'+})\in {\mathbb N}^n\times \mathbb{N}^n \,\, :\ 
  \norm{\dvec^{'-}}_1  = \norm{\dvec^{'+}}_1,\quad 
       \norm{\dvec^{'-}-\dvec^-}_1 + \norm{\dvec^{'+}-\dvec^+}\le k\}.\]

The following analogue of Lemma~\ref{lem:transitive} is proved using the
same ideas (proof omitted).

\begin{lemma}\label{lem:transitive-directed}
Suppose that every digraphical degree sequence 
$(\dvec^{'-},\dvec^{'+})\in \N_{10}(\dvec^-,\dvec^+)$ is $(2,\alpha)$-stable. 
Then $\vecdvec$ is $(12,6\alpha)$-stable. 
\end{lemma}

Now we state the sufficient condition for (2,2)-stability, which is the
directed analogue of Theorem~\ref{thm:main}.  Luckily, the switching argument
from Theorem~\ref{thm:main} can be fairly easily adapted to the directed setting.

\begin{theorem}\label{thm:main-directed}
If $\vecdvec$ satisfies $M> J^-+ J^+ +3(\Delta^- + \Delta^+) +3$
 then $\vecdvec$ is $(2,2)$-stable.
\end{theorem}

\begin{proof}
Let $(\dvec^{'-},\dvec^{'+})\in\N_2(\dvec^-,\dvec^+)\setminus \vecdvec$.
There are only four cases:  either $\dvec^{'-} = \dvec^- + \boldsymbol{e}_i$ and 
$\dvec^{'+} = \dvec^+ + \boldsymbol{e}_j$, or $\dvec^{'-} = \dvec^- - \boldsymbol{e}_i$ and 
$\dvec^{'+} = \dvec^+ - \boldsymbol{e}_j$, or $\dvec^{'-}=\dvec+\boldsymbol{e}_i-\boldsymbol{e}_j$, 
or $\dvec^{'+}=\dvec+\boldsymbol{e}_i-\boldsymbol{e}_j$, for some $i,j\in [n]$.
These cases are the directed analogues of the cases from the proof of
Theorem~\ref{thm:main}, presented in Section~\ref{s:deferred}.
We do not give full details but summarise the calculations.

In Case 1, the directed $(i^-,j^-)$-switching is shown in Figure~\ref{f:1a1b-directed},
with the subcase $i\neq j$ on the left and $i=j$ on the right.
This switching converts a graph $G'\in \Omega(\dvec^{'-}+\boldsymbol{e}_i,\dvec^{'+}
+ \boldsymbol{e}_j)$ to a graph $G\in \Omegad$. 
Note that the arcs  alternate in orientation, as well as in their presence/absence,
at each vertex other than $i,j$.  Again we allow $u_1=u_4$.

\begin{figure}[ht!]
\begin{center}
\begin{tikzpicture}[scale=0.8]
\draw [->,very thick] (0,2) -- (0,.2);
\draw [->,very thick] (2,4) -- (0.2,4);
\draw [->,very thick] (2,0) -- (2,1.8);
\draw [->,very thick,dashed] (0,2) -- (0,3.8); 
\draw [->,very thick,dashed] (2,4) -- (2,2.2);
\draw [fill] (0,0) circle (0.12); \draw [fill] (0,2) circle (0.12);
\draw [fill] (0,4) circle (0.12); \draw [fill] (2,0) circle (0.12);
\draw [fill] (2,2) circle (0.12); \draw [fill] (2,4) circle (0.12);
\node [left] at (-0.2,0) {$i$}; \node [right] at (2.2,0) {$j$};
\node [left] at (-0.1,2) {$u_1$}; \node [right] at (2.1,2) {$u_4$};
\node [left] at (-0.1,4) {$u_2$}; \node [right] at (2.1,4) {$u_3$};
\draw [->,line width = 1mm] (3.2,2) -- (4.0,2);
\begin{scope}[shift={(5.25,0)}]
\draw [->,very thick,dashed] (0,2) -- (0,.2);
\draw [->,very thick,dashed] (2,4) -- (0.2,4);
\draw [->,very thick,dashed] (2,0) -- (2,1.8);
\draw [->,very thick] (0,2) -- (0,3.8); 
\draw [->,very thick] (2,4) -- (2,2.2);
\draw [fill] (0,0) circle (0.12); \draw [fill] (0,2) circle (0.12);
\draw [fill] (0,4) circle (0.12); \draw [fill] (2,0) circle (0.12);
\draw [fill] (2,2) circle (0.12); \draw [fill] (2,4) circle (0.12);
\node [left] at (-0.2,0) {$i$}; \node [right] at (2.2,0) {$j$};
\node [left] at (-0.1,2) {$u_1$}; \node [right] at (2.1,2) {$u_4$};
\node [left] at (-0.1,4) {$u_2$}; \node [right] at (2.1,4) {$u_3$};
\end{scope}
%%%%%%
\draw [-] (9.25,-0.5) -- (9.25,4.5); 
%%%%%%
\begin{scope}[shift={(11,0)}]
\draw [->,very thick] (0,2) -- (0.85,0.15);
\draw [->,very thick] (1,0) -- (1.87,1.84); 
\draw [->,very thick] (2,4) -- (0.2,4); 
\draw [->,very thick,dashed] (0,2) -- (0,3.8); 
\draw [->,very thick,dashed] (2,4) -- (2,2.2);
\draw [fill] (1,0) circle (0.12); \draw [fill] (0,2) circle (0.12);
\draw [fill] (0,4) circle (0.12);
\draw [fill] (2,2) circle (0.12); \draw [fill] (2,4) circle (0.12);
\node [left] at (0.8,0) {$i$}; 
\node [left] at (-0.1,2) {$u_1$}; \node [right] at (2.1,2) {$u_4$};
\node [left] at (-0.1,4) {$u_2$}; \node [right] at (2.1,4) {$u_3$};
\draw [->,line width = 1mm] (3.2,2) -- (4.0,2);
\begin{scope}[shift={(5.25,0)}]
\draw [->,very thick,dashed] (0,2) -- (0.85,0.15);
\draw [->,very thick,dashed] (1,0) -- (1.87,1.84); 
\draw [->,very thick,dashed] (2,4) -- (0.2,4); 
\draw [->,very thick] (0,2) -- (0,3.8); 
\draw [->,very thick] (2,4) -- (2,2.2);
\draw [fill] (1,0) circle (0.12); \draw [fill] (0,2) circle (0.12);
\draw [fill] (0,4) circle (0.12); 
\draw [fill] (2,2) circle (0.12); \draw [fill] (2,4) circle (0.12);
\node [left] at (0.8,0) {$i$}; 
\node [left] at (-0.1,2) {$u_1$}; \node [right] at (2.1,2) {$u_4$};
\node [left] at (-0.1,4) {$u_2$}; \node [right] at (2.1,4) {$u_3$};
\end{scope}
\end{scope}
\end{tikzpicture}
\caption{The directed $(i^-,j^-)$-switching, when $i\neq j$ (left) and when $i=j$ (right).}
\label{f:1a1b-directed}
\end{center}
\end{figure}
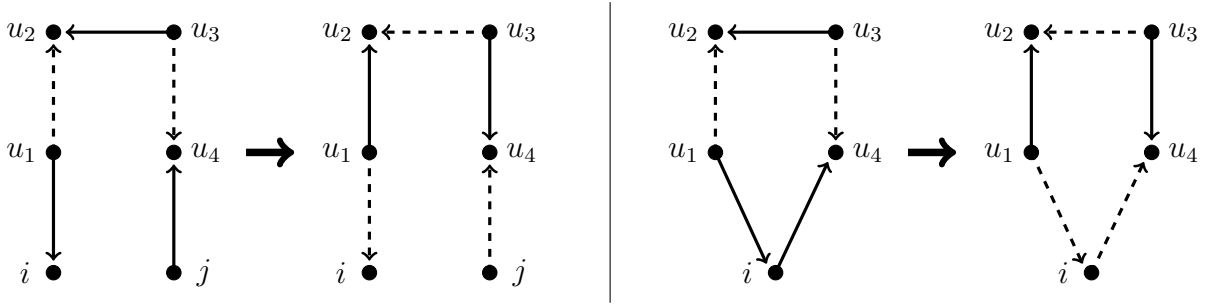

Let $f_{(i^-,j^-)}(G')$ be the number of ways to perform an $(i^-,j^-)$-switching
to $G'\in\Omega(\dvec^-+\boldsymbol{e}_i,\dvec^+ + \boldsymbol{e}_j)$.
There are $(d_i^-+1)(d_j^++1)$ ways to choose $u_1$ and $u_4$.
Then there are $M(\dvec^-,\dvec^+) + 1$ choices for the arc $(u_3,u_2)$, but we must 
subtract $4(\Delta^-(\dvec^-) + \Delta^+(\dvec^+) + 2)$
for the possible vertex coincidences.  Furthermore we must subtract the
number of choices where there are no vertex coincidences but the arc $(u_1,u_2)$
is present, or the arc $(u_3,u_4)$ is present.  The number of choices 
such that arc $(u_1,u_2)$ is present is at most
\[ \sum_{\ell=1}^{d^+_{u_1}-1} (d^-_{\mu(\ell)} - 1) \leq 
  \sum_{\ell=1}^{\Delta^+(\dvec^+)-1} (d^-_{\mu(\ell)} - 1) \leq 
J^+(\dvec^-,\dvec^+) -\Delta^+(\dvec^+)\]
and similarly, the number of choices such that $(u_3,u_4)$ is present is
at most $J^-(\dvec^-,\dvec^+) - \Delta^-(\dvec^-)$.  Therefore
\[ f_{(i^-,j^-)}(G')\geq
 (d_i^-+1)(d_j^++1)\Big( M - J^- -
   J^+- 3\Delta^- - 3\Delta^+ - 3\Big),\]
which is positive by the assumption of the theorem.
The number of ways to create $G\in\Omegad$ using a $(i^-,j^-)$-switching is
at most $M(\dvec^-,\dvec^+)^2$, and we conclude that
\[ \frac{|\Omega(\dvec^- + \boldsymbol{e}_i,\dvec^+ + \boldsymbol{e}_j)|}
        {|\Omegad|}
    \leq M(\dvec^-,\dvec^+)^2.\]
The argument when $i=j$ is similar and also leads to an upper bound of $M(\dvec^-,
\dvec^+)^2$.  A similar analysis shows that $M(\dvec^-,\dvec^+)^2$ is also
an upper bound in the other cases, by adapting the switchings and arguments from
Theorem~\ref{thm:main}.
\end{proof}

Now we can use Theorem~\ref{thm:main-directed} to prove 
Theorem~\ref{thm:stable-directed}.

\begin{proof}[Proof of Theorem~\ref{thm:stable-directed}]
First suppose that $\dvec$ satisfies (\ref{big-condition-directed}) and let 
$\dvec'\in\mathcal{N}_{10}(\dvec)$.
Then \[ \max_{i\in [n]} \,\{|d^{'-}_i-d^-_i|, |d^{'+}_i-d^+_i|\} \le 5,\] and thus 
$M(\dvec^{'-},\dvec^{'+})\ge M(\dvec^-,\dvec^+)-5$. 
Also
\[ \Delta(\dvec^{'-}) + \Delta(\dvec^{'+}) \leq
    \Delta(\dvec^-) + \Delta(\dvec^+) + 10\]
and
\[J^-(\dvec^{'-},\dvec^{'+}) + J^+(\dvec^{'-},\dvec^{'+})\le 
      J^-(\dvec^-,\dvec^+) + J^+(\dvec^-,\dvec^+) 
   + 5 \big(\Delta(\dvec^-) + \Delta(\dvec^+)\big) +10. 
\]
Applying Theorem~\ref{thm:main-directed} to $\dvec'$ shows that $\dvec'$ is
(2,2)-stable, and therefore $\dvec$ is (12,12)-stable, 
by Lemma~\ref{lem:transitive-directed}.  
This completes the proof.
\end{proof}

\subsection{Digraphical power-law sequences}

Recall the definition of power-law distribution-bounded and power-law density-bounded degree sequences
from Sections~\ref{ss:power-density} and~\ref{ss:power-distribution}. Let $D: {\mathbb N}\to {\mathbb N}$. 
We say that a sequence $\avec\in\mathbb{N}^n$ is \emph{$D(n)$-bounded} if
every component of $\avec$ is at most $D(n)$.
Now consider digraphical sequences $(\dvec^-,\dvec^+)$ where $\dvec^+$ is power-law distribution-bounded, or power-law density bounded, with exponent $\gamma>1$, whereas $\dvec^-$ is $\Delta^-$-bounded (in the sense defined above), where $\Delta^-$ is some function of $n$.
Such degree sequences include models where $\dvec^+$ is composed of i.i.d.\ copies of power-law random variables, and $\dvec^-$ is composed of i.i.d.\ copies of truncated power-law random variables~\cite{Mossa}.  Directed degree sequences of this type are important because
some social networks such as Twitter are directed graphs~\cite{Aparicio}, where the indegrees (i.e.\ the number of followers) 
appear to follow 
a power law with parameter below 2, whereas the majority of the outdegrees follows a power law, but with a much smaller maximum outdegree compared with the maximum indegree. 
 In the next two theorems we show that random directed graphs with such degree sequences can be sampled 
in polynomial time using the directed switch chain. 
We say a digraphical degree sequence is \emph{nontrivial} if there is no $i$ such that $d_i^+=d_i^-=0$. If $(\dvec^-,\dvec^+)$ is nontrivial then $M(\dvec^-,\dvec^+)=\frac{1}{2}(\norm{\dvec^-}_1+\norm{\dvec^+}_1)\ge n$.

\begin{theorem}\label{thm:directed-powerlaw}
Let $(\dvec^-,\dvec^+)$ be a digraphical, switch-irreducible nontrivial degree sequence
such that $\dvec^+$ is power-law density-bounded with parameter 
$\gamma>1$, and $\Delta^-=o(n^{(\gamma-1)/\gamma})$.
Then the directed switch chain on $\Omega(\dvec^-,\dvec^+)$
mixes in polynomial time, 
with mixing time bounded by~\emph{(\ref{mixingbound-directed})}.  
\end{theorem}

\begin{proof} By definition, $\max\{J^-, J^+\} \le \Delta^- \Delta^+$. 
Since $\Delta^+=O(n^{1/\gamma})$ and $M=\Theta(n)$ by the definition of power-law density-bounded sequence, it follows that $(\dvec^-,\dvec^+)$ is $(12,12)$-stable by Theorem~\ref{thm:stable-directed}. The assertion then follows by Theorem~\ref{thm:mixing2-directed}.
\end{proof}

\begin{theorem}
Let $(\dvec^-,\dvec^+)$ be a digraphical, switch-irreducible directed degree sequence
such that $\dvec^+$ is power-law distribution-bounded with parameter 
$\gamma>2$, and $\Delta^-=o(n^{(\gamma-2)/(\gamma-1)})$.
Then the directed switch chain on $\Omega(\dvec^-,\dvec^+)$
mixes in polynomial time, 
with mixing time bounded by~\emph{(\ref{mixingbound-directed})}.  
\end{theorem}

\begin{proof} The proof is the same as that of Theorem~\ref{thm:directed-powerlaw} except that with 
power-law distribution-bounded sequences we have  $\Delta^+=O(n^{1/(\gamma-1)})$ and $M=\Theta(n)$.
\end{proof}


\begin{thebibliography}{99}

\bibitem{AK}
G.~Amanatidis and P.~Kleer,
Rapid mixing of the switch Markov chain for strongly stable degree
sequences and 2-class joint degree matrices,
in Proceedings of the 30th Annual ACM--SIAM Symposium on
Discrete Algorithms, 2019, ACM--SIAM, New York--Philadelphia, pp.~966--985.

\bibitem{Aparicio} S.~Aparicio, J.~Villaz{\'o}n-Terrazas and G.~{\'A}lvarez, 
A model for scale-free networks: application to twitter,
\emph{Entropy},
  {\bf 17(8)} (2015), {5848--5867}.
  
\bibitem{AGW}
A.~Arman, P.~Gao and N.~Wormald,
Fast uniform generation of random graphs with given degree sequences.
Preprint (2019). \arxiv{1905.03446}

\bibitem{BKS}
M.~Bayati, J.H.~Kim and A.~Saberi,
A sequential algorithm for generating random graphs,
\emph{Algorithmica} {\bf 58} (2010), 860--910.

\bibitem{BMH}
A.~Berger and M.~M{\" u}ller--Hannemann,
Uniform sampling of digraphs with a fixed degree sequence,
in \emph{Graph Theoretic Concepts in Computer Science},
Lecture Notes in Computer Science vol.~6410, Springer, Berlin,
2010, pp.~220--231.


\bibitem{ChungLu}
F.~Chung and L.~Lu, The volume of the giant component of a random graph with
given expected degrees, \emph{SIAM Journal on Discrete Mathematics} {\bf 20} (2006), 395--411.

\bibitem{CDG}
C.~Cooper, M.E.~Dyer and C.~Greenhill,
Sampling regular graphs and a peer-to-peer network,
\emph{Combinatorics, Probability and Computing} {\bf 16}
(2007), 557--593.

\bibitem{hungarians}
P.L.~Erd{\H o}s, C.~Greenhill, T.R.~Mezei, I.~Miklos, D.~Solt{\' e}sz, L.~Soukup,
The mixing time of the switch Markov chains: a unified approach.
Preprint (2019). \arxiv{1903.06600}

\bibitem{non-P-stable}
P.L.~Erd{\H o}s, E.~Gy{\H o}ri, T.R.~Mezei, I.~Mikl{\' o}s and D.~Solt{\' e}sz,
A non-P-stable class of degree sequences for which the swap Markov chain is
rapidly mixing. Preprint (2019).  \arxiv{1909.02308}

\bibitem{EKMS}
P.L.~Erd{\H o}s, S.Z.~Kiss, I.~Mikl{\' o}s and L.~Soukup,
Approximate counting of graphical realizations,
\emph{PLoS ONE} {\bf 10} (2015), e0131300. 

\bibitem{EMMS}
P.L.~Erd{\H o}s, T.R.~Mezei, I.~Miklos and D.~Solt{\' e}sz,
Efficiently sampling the realizations of bounded, irregular degree sequences
of bipartite and directed graphs,
\emph{PLoS ONE} {\bf 13(8)} (2018), e0201995.

\bibitem{EMT}
P.L.~Erd{\H o}s, I.~Mikl{\' o}s and Z.~Toroczkai,
New classes of degree sequences with fast mixing swap Markov chain sampling,
\emph{Combinatorics, Probability and Computing} {\bf 27} (2018), 186--207.

 \bibitem{G57} 
D.~Gale, 
A theorem on flows in networks, 
\emph{Pacific Journal of Mathematics} {\bf 7}  (1957), 1073--1082. 

\bibitem{GW}
P.~Gao and N.~Wormald,
Uniform generation of random regular graphs,
in Proceedings of the 56th Annual IEEE Symposium on
Foundations of Computer Science, 2015, IEEE, Los Alamitos, 
pp.~1218--1230.

\bibitem{GW-power}
P.~Gao and N.~Wormald,
Enumeration of graphs with a heavy-tailed degree sequence,
\emph{Advances in Mathematics} {\bf 287} (2016), 412--450.

\bibitem{Gao-power} P.~Gao and N.~Wormald, Uniform generation of random graphs with power-law degree sequences, in Proceedings of the Twenty-Ninth Annual ACM-SIAM Symposium on Discrete Algorithms, Society for Industrial and Applied Mathematics, 2018.

\bibitem{directed}
C.~Greenhill,
A polynomial bound on the mixing time of a Markov chain for 
sampling regular directed graphs, 
\emph{Electronic Journal of Combinatorics} {\bf 18} (2011), \#P234.

\bibitem{GS}
C.~Greenhill and M.~Sfragara, The switch Markov chain for sampling irregular graphs
and digraphs, \emph{Theoretical Computer Science} {\bf 719} (2018), 1--20.

\bibitem{JS90}
M.~Jerrum and A.~Sinclair, Fast uniform generation of regular graphs,
\emph{Theoretical Computer Science} {\bf 73} (1990), 91--100.

\bibitem{JMS}
M.~Jerrum, A.~Sinclair and B.~McKay, When is a graphical sequence stable?,
in \emph{Random Graphs, Vol.~2 (Pozn{\' a}n, 1989)}, Wiley, New York, 1992,
pp.~101--115.

\bibitem{KTV}
R.~Kannan, P.~Tetali and S.~Vempala,
Simple Markov-chain algorithms for generating bipartite graphs and
tournaments,
\emph{Random Structures and Algorithms} {\bf 14} (1999),
293--308.

\bibitem{KV}
J.H.~Kim and V.~Vu,
Generating random regular graphs,
\emph{Combinatorica} {\bf 26} (2006), 683--708.

\bibitem{lamar}
M.D.~LaMar, 
On uniform sampling simple directed graph realizations of degree sequences,
Preprint (2009). \arxiv{0912.3834}

\bibitem{mckay-line}
B.D.~McKay,
Asymptotics for 0-1 matrices with prescribed line sums,
\emph{Enumeration and Design}, Academic Press, Toronto, 1984,
pp.~225--238.

\bibitem{McKW90}
B.D.~McKay and N.C.~Wormald,
Uniform generation of random regular graphs of moderate degree,
\emph{Journal of Algorithms} {\bf 11} (1990), 52--67.

\bibitem{McKW91}
B.D.~McKay and N.C.~Wormald, Asymptotic enumeration by degree sequence of
graphs with degrees $o(n^{1/2})$,
\emph{Combinatorica} {\bf 11} (1991), 369--382.

\bibitem{MES}
I.~Mikl{\' o}s, P.L.~Erd{\H o}s and L.~Soukup,
Towards random uniform sampling of bipartite graphs with given degree sequence,
\emph{Electronic Journal of Combinatorics} {\bf 20(1)}, 2013, \#P16.

\bibitem{Mossa} S.~Mossa, M.~Barthelemy, H. E.~Stanley and L. A. N.~Amaral, 
Truncation of power law behavior in ``scale-free'' network models due to information filtering,
\emph{Physical Review Letters} {\bf 88(13)} (2002): 138701.

\bibitem{newman}
M.E.~Newman, The structure and function of complex networks,
\emph{SIAM Review} {\bf 45.2} (2003), 167--256.


\bibitem{pet} 
J.~Petersen, Die Theorie der regularen Graphen, 
\emph{Acta Mathematica} {\bf 15} (1891), 193--220.

\bibitem{sinclair}
A.~Sinclair, 
Improved bounds for mixing rates of Markov chains
and multicommodity flow, 
\emph{Combinatorics, Probability and Computing}
{\bf 1} (1992), 351--370.

\bibitem{SW}
A.~Steger and N.~Wormald,
Generating random regular graphs quickly,
\emph{Combinatorics, Probability and Computing} {\bf 8} (1999), 377--396.

\bibitem{vdH}
R.~Van Der Hofstad, Random graphs and complex networks,
\emph{Cambridge Series in Statistical and Probabilistic Mathematics}
{\bf 43} (2016).

\bibitem{zhao}
J.Y.~Zhao, 
Expand and Contract: Sampling graphs with given degrees and other combinatorial 
families, Preprint, 2013. 
\arxiv{1308.6627}

\end{thebibliography}
\end{document}